\documentclass{article}

\usepackage[a4paper]{geometry}
\usepackage{amsmath}
\usepackage{amssymb}
\usepackage{amsthm}
\usepackage{multirow}
\usepackage{enumerate}
\usepackage{longtable}
\usepackage{caption2}
\usepackage{cleveref}

\allowdisplaybreaks[1]
\newcommand{\ra}[1]{\renewcommand{\arraystretch}{#1}}

\newtheorem{theorem}{Theorem}
\newtheorem{corollary}[theorem]{Corollary}
\newtheorem{lemma}[theorem]{Lemma}
\newtheorem{proposition}[theorem]{Proposition}
\newtheorem{definition}[theorem]{Definition}
\newtheorem{conjecture}{Conjecture}
\newtheorem{question}{Question}

\theoremstyle{definition}
\newtheorem{example}[theorem]{Example}
\newtheorem{remark}[theorem]{Remark}

\numberwithin{theorem}{section}
\numberwithin{equation}{section}
\numberwithin{table}{section}

\newcommand{\bb}{\mathbb}
\newcommand{\Z}{\bb{Z}}

\newcommand{\ZN}{\Z_N}

\newcommand{\Zp}{\Z_p}
\newcommand{\Q}{\bb{Q}}
\newcommand{\R}{\bb{R}}
\newcommand{\C}{\bb{C}}

\newcommand{\cC}{\mathcal{C}}

\newcommand{\cP}{\mathcal{P}}
\newcommand{\cQ}{\mathcal{Q}}

\newcommand{\cT}{\mathcal{T}}

\newcommand{\al}{\alpha}

\newcommand{\be}{\beta}
\newcommand{\de}{\delta}
\newcommand{\De}{\Delta}

\newcommand{\Ga}{\Gamma}
\newcommand{\la}{\lambda}
\newcommand{\La}{\Lambda}

\newcommand{\Sig}{\Sigma}
\newcommand{\sig}{\sigma}
\newcommand{\ze}{\zeta}
\newcommand{\zp}{\zeta_{p}}
\newcommand{\zpt}{\zeta_{p^t}}
\newcommand{\zn}{\zeta_{n}}

\newcommand{\pr}{\prime}

\newcommand{\sm}{\setminus}
\newcommand{\es}{\emptyset}

\newcommand{\ol}{\overline}
\newcommand{\lan}{\langle}
\newcommand{\ran}{\rangle}

\newcommand{\wti}{\widetilde}
\newcommand{\wh}{\widehat}
\newcommand{\lf}{\lfloor}
\newcommand{\rf}{\rfloor}

\newcommand{\Aut}{\text{Aut}}
\newcommand{\Gal}{\text{Gal}}
\newcommand{\GR}{\text{GR}}

\newcommand{\orb}{\text{orb}}

\newcommand{\Span}{\text{span}}

\newcommand{\Tr}{\text{Tr}}

\newcommand{\Fpm}{\mathbb{F}_{p^m}}

\newcommand{\zl}{\zeta_l}
\newcommand{\Zl}{\Z_l}

\newcommand{\AGL}{\text{AGL}}

\newcommand{\covol}{\text{covol}}
\newcommand{\vol}{\text{vol}}

\long\def\symbolfootnote[#1]#2{\begingroup%
\def\thefootnote{\fnsymbol{footnote}}\footnote[#1]{#2}\endgroup}

\begin{document}

\title{Formal Duality in Finite Abelian Groups}
\author{Shuxing Li \and Alexander Pott \and Robert Sch\"{u}ler}
\date{}
\maketitle

\symbolfootnote[0]{
S.~Li and A.~Pott are with the Faculty of Mathematics, Otto von Guericke University Magdeburg, 39106 Magdeburg , Germany (e-mail: shuxing.li@ovgu.de, alexander.pott@ovgu.de).
\par
R.~Sch\"{u}ler is the with Institute of Mathematics, University of Rostock, 18051 Rostock, Germany (e-mail: robert.schueler2@uni-rostock.de).
}

\begin{abstract}
Inspired by an experimental study of energy-minimizing periodic configurations in Euclidean space, Cohn, Kumar and Sch\"urmann proposed the concept of formal duality between a pair of periodic configurations, which indicates an unexpected symmetry possessed by the energy-minimizing periodic configurations. Later on, Cohn, Kumar, Reiher and Sch\"urmann translated the formal duality between a pair of periodic configurations into the formal duality of a pair of subsets in a finite abelian group. This insight suggests to study the combinatorial counterpart of formal duality, which is a configuration named formally dual pair. In this paper, we initiate a systematic investigation on formally dual pairs in finite abelian groups, which involves basic concepts, constructions, characterizations and nonexistence results. In contrast to the belief that primitive formally dual pairs are very rare in cyclic groups, we construct three families of primitive formally dual pairs in noncyclic groups. These constructions enlighten us to propose the concept of even sets, which reveals more structural information about formally dual pairs and leads to a characterization of rank three primitive formally dual pairs. Finally, we derive some nonexistence results about primitive formally dual pairs, which are in favor of the main conjecture that except two small examples, no primitive formally dual pair exists in cyclic groups.

\smallskip
\noindent \textbf{Keywords.} Character sum, finite abelian group, energy minimization, even set, formal duality, formally dual pair, formally dual set, lattice, periodic configuration, relative difference set, skew Hadamard difference set.

\noindent {{\bf Mathematics Subject Classification\/}: 05B40, 52C17.}
\end{abstract}

\section{Energy-minimizing configuration and formal duality}

Let $\cC$ be a particle configuration in the Euclidean space $\R^n$ and $f: \R^n \rightarrow \R$ be a potential function. The energy minimization problem, with respect to the potential function $f$, aims to find configurations $\cC \subset \R^n$ with a fixed density, which possesses the minimal energy measured by $f$. The famous sphere packing problem can be viewed as an extremal case of the energy minimization problem \cite[p. 123]{CKRS}. In general, the energy minimization problem is extremely difficult and very few theoretical results are available.

In \cite{CKRS,CKS}, the authors considered a special case of the energy minimization problem, in which the particle configurations under consideration are restricted to the so called periodic configurations. Let $\La$ be a lattice in $\R^n$, which is a discrete subgroup of $\R^n$. A \emph{periodic configuration} $\cP=\bigcup_{j=1}^N (v_j+\La)$ is a union of finitely many translations $v_1, v_2, \ldots, v_N$ of the lattice $\La$. The \emph{density} of $\cP$ is defined to be $\de(\cP)=N/\covol(\La)$, where $\covol(\La)=\vol(\R^n/\La)$ is the volume of a fundamental domain of $\La$. In \cite{CKS}, the authors conducted an experimental study of energy-minimizing periodic configurations, with respect to the Gaussian potential function $G_c(r)=e^{-\pi cr^2}$, for some $c>0$. The numerical simulations in \cite{CKS} concerned a pair of Gaussian potential functions $G_c$ and $G_{1/c}$, for which a pair of suspected energy-minimizing periodic configurations was obtained. Surprisingly, this pair of periodic configurations possesses a strong symmetry named formal duality \cite[Section VI]{CKS}, in the sense that they behave like the dual of each other.

Note that up to a scalar multiple, $G_{1/c}$ is equal to the Fourier transformation of $G_{c}$. Thus, formal duality can be stated with respect to a well-behaved potential function $f: \R^n \rightarrow \R$ and its Fourier transformation
\begin{equation}\label{def-Fourier}
\wh{f}(y)=\int_{\R^n}f(x)e^{-2\pi i\lan x,y \ran}dx,
\end{equation}
where $\lan \cdot, \cdot \ran$ is the inner product in $\R^n$. The function $f$ is well-behaved in the sense that the convergence of its Fourier transformation \eqref{def-Fourier} is guaranteed. For instance, we may assume that $f$ is a Schwartz function. Roughly speaking, for a Schwartz function $f: \R^n \rightarrow \R$, the absolute value $|f(x)|$ decreases rapidly enough to $0$ when $|x|$ goes to infinity, so that its Fourier transformation $\wh{f}$ is well-defined. For a Schwartz function $f: \R^n \rightarrow \R$ and a periodic configuration $\cP=\bigcup_{j=1}^N(v_j+\La)$ with respect to a lattice $\La \subset \R^n$, define the \emph{average pair sum} of $f$ over $\cP$ as
$$
\Sig_f(\cP)=\frac{1}{N}\sum_{j,\ell=1}^N\sum_{x \in \La}f(x+v_j-v_\ell),
$$
which measures the average energy among the differences of pairs in $\cP$. Now we proceed to introduce the definition of formal duality.

\begin{definition}{\rm \cite[Definition 2.1]{CKRS}}
Two periodic configurations $\cP$ and $\cQ$ are formally dual to each other, if for each Schwartz function $f: \R^n \rightarrow \R$,
\begin{equation}\label{def-formaldual}
\Sig_f(\cP)=\de(\cP)\Sig_{\wh{f}}(\cQ),
\end{equation}
where $\wh{f}$ is the Fourier transformation of $f$.
\end{definition}
Let $\La \subset \R^n$ be a lattice with a basis containing $n$ vectors. Consider a special case of periodic configurations, where $\cP$ is a lattice $\La$ and $\cQ$ is the dual lattice
$$
\La^*=\{x \in \R^n \mid \lan x, y \ran \in \Z, \forall y \in \La \},
$$
in which $\lan \cdot, \cdot \ran$ is the inner product in $\R^n$. Then \eqref{def-formaldual} corresponds to the Poisson summation formula
$$
\sum_{x \in \La} f(x)=\frac{1}{\covol(\La)}\sum_{y \in \La^*}\wh{f}(y),
$$
which describes the duality between a lattice $\La$ and its dual $\La^*$. Thus, the formal duality can be viewed as an extension of the duality between a lattice and its dual.

For a periodic configuration $\cP=\bigcup_{j=1}^N(v_j+\La)$ with respect to a lattice $\La$, whether it has a formally dual or not, depends only on the lattice $\La$ and the translations $v_j$, not on how $\cP$ is embedded in $\R^n$ \cite[p. 127]{CKRS}. This crucial fact suggests that formal duality can be interpreted as a phenomenon in finite abelian groups.

Let $\cP=\bigcup_{j=1}^N(v_j+\La)$ and $\cQ=\bigcup_{j=1}^M(w_j+\Ga)$ be two periodic configurations with respect to two lattices $\La$ and $\Ga$ in $\R^n$. Without loss of generality, we assume $0 \in \cP$ and $0 \in \cQ$. If $\cP$ and $\cQ$ are formally dual to each other, then by \cite[Corollary 2.6]{CKRS}, $\cP \subset \Ga^*$ and $\cQ \subset \La^*$. Thus, $\La$ is a subgroup of $\Ga^*$ and $\Ga$ is a subgroup of $\La^*$, so that $\Ga^*/\La$ and $\La^*/\Ga$ are finite abelian groups \cite[Theorem 2.6]{F}. Furthermore, $\cP$ can be associated with a subset $S=\{v_j \mid 1 \le j \le N\}$ of size $N$ in the finite abelian group $\Ga^*/\La$ and $\cQ$ can be associated with a subset $T=\{w_j \mid 1 \le j \le M\}$ of size $M$ in the finite abelian group $\La^*/\Ga$ \cite[p. 129]{CKRS}. By identifying $\Ga^*/\La$ with a finite abelian group $G$ and $\La^*/\Ga$ with its character group $\wh{G}$, the following proposition shows that the formal duality of $\cP$ and $\cQ$ is equivalent to a property between a pair of subsets $S \subset G$ and $T \subset \wh{G}$.

\begin{proposition}\label{prop-translation}{\rm \cite[Theorem 2.8]{CKRS}}
Let $\cP=\bigcup_{j=1}^N(v_j+\La)$ and $\cQ=\bigcup_{j=1}^M(w_j+\Ga)$ be two periodic configurations, associated with a subset $S=\{v_j \mid 1 \le j \le N\} \subset G=\Ga^*/\La$ and a subset $T=\{w_j \mid 1 \le j \le M\} \subset \wh{G}=\La^*/\Ga$. Then $\cP$ and $\cQ$ are formally dual if and only if for every $y \in \wh{G}$,
\begin{equation}\label{eqn-convert}
\left|\frac{1}{N}\sum_{i=1}^N \lan v_i,y \ran\right|^2=\frac{1}{M}|\{(j,\ell) \mid 1 \le j,\ell \le M, y=w_jw_\ell^{-1}\}|,
\end{equation}
where $\lan v_i,y \ran$ denotes the evaluation of the character $y$ at the element $v_i$.
\end{proposition}

Note that $\Ga^*$ is a finitely-generated free abelian group and $\La$ is a subgroup of $\Ga^*$. By the Stacked Basis Theorem \cite[Theorem 2.6]{F}, each abelian group can be realized as a quotient between a lattice $\Ga^*$ and a proper subgroup $\La$. Hence, the above proposition shows that the formal duality between two periodic configurations can be translated into a combinatorial setting, which considers the formal duality of a pair of subsets $S \subset G$ and $T \subset \wh{G}$. A pair of subsets $S \subset G$ and $T \subset \wh{G}$ is called a formally dual pair (see Definition~\ref{def-ori}), if they satisfy \eqref{eqn-convert}. A subset $S \subset G$ is called a formally dual set, if there exists a subset $T \subset \wh{G}$, such that $S$ and $T$ form a formally dual pair. As suggested by the experimental results in \cite{CKS}, the knowledge about formally dual pairs may bring us a better understanding on energy-minimizing periodic configurations and provide potential candidates of energy-minimizing periodic configurations.

There have been some known results concerning formally dual pairs in finite cyclic groups \cite{Mali,Sch,Xia}. In this paper, we initiate a systematic investigation on formally dual pairs in finite abelian groups. In Section~\ref{sec2}, we give a detailed account of formally dual pairs and summarize some known results. We introduce an equivalent description of formally dual pairs and formally dual sets, where the latter one contain much information about corresponding formally dual pairs (Definition~\ref{def-iso}). Moreover, the concept of equivalence between two formally dual sets is proposed (Definition~\ref{def-equiv}). In Section~\ref{sec3}, we present three constructions of primitive formally dual pairs in elementary abelian groups and in products of Galois rings (Theorems~\ref{thm-RDS}, \ref{thm-GRDS} and~\ref{thm-SHDS}). Observing the close connection between two of these three constructions and relative difference sets, we propose the concept of even sets in Section~\ref{sec4}, which offers a new model for the study of formally dual sets (Definition~\ref{def-evenset}). Even sets provide some new insight on the structure of primitive formally dual sets, from which the relation between primitive formally dual sets and $(n,n,n,1)$ relative difference sets is revealed (Theorem~\ref{thm-rankthree}). Moreover, even sets provide a new viewpoint on the main conjecture concerning primitive formally dual sets in cyclic groups (Question~\ref{ques-rankthree}). In Section~\ref{sec5}, we derive some nonexistence results about primitive formally dual pairs in favor of the main conjecture (Conjecture~\ref{conj-main}). Applying all known nonexistence results, we demonstrate that there are only three open cases for primitive formally dual pairs in $\Z_N$, with $N \le 1000$ (Remark~\ref{rem-cycopen}). For every abelian group of order up to $63$, a computer search either lists all inequivalent primitive formally dual sets in it or establishes a nonexistence result (Remark~\ref{rem-search}). Section~\ref{sec6} consists of some concluding remarks.

\section{Fundamental facts about formally dual pairs}\label{sec2}

Throughout the paper, we always consider finite abelian groups. For a group $G$, its exponent $\exp(G)$ is the order of the largest cyclic subgroup contained in $G$. The automorphism group of $G$ is denoted by $\Aut(G)$. We use $1$ to denote the identity element in $G$ when the operation of $G$ is written multiplicatively and use $0$ when the operation is written additively. The specific representation of the identity of $G$ is easy to see from the context. Let $G$ be a group and $A$ a subset of $G$. For each $y \in G$, define the \emph{weight enumerator} of $A$ at $y$ as
$$
\nu_A(y)=|\{(a_1,a_2) \in A \times A \mid y=a_1a_2^{-1}\}|.
$$

We use $\Z[G]$ to denote the group ring. Given an integer $i$ and $A=\sum_{g \in G} a_gg \in \Z[G]$, we define $A^{(i)}=\sum_{g \in G}a_gg^i \in \Z[G]$. For $A \in \Z[G]$, if each coefficient of $A$ is nonnegative, we use $\{A\}$ to denote the underlying subset of $G$ corresponding to $A$ and $[A]$ the multiset corresponding to $A$. For $A \in \Z[G]$ and $g \in G$, we use $[A]_g$ to denote the coefficient of $g$ in $A$. A \emph{character} $\chi$ of $G$ is a group homomorphism from $G$ to the multiplicative group of the complex field $\mathbb{C}$. For a group $G$, we use $\wh{G}$ to denote its character group. A character $\chi \in \wh{G}$ is \emph{principal}, if $\chi(g)=1$ for each $g \in G$. A character $\chi \in \wh{G}$ is principal on a subgroup $H \leqslant G$, if $\chi(h)=1$ for each $h \in H$. Let $H$ be a subgroup of $G$, define $H^{\perp}=\{\chi \in \wh{G} \mid \chi \mbox{ principal on $H$}\}$. For $\chi \in \wh{G}$ and $A \in \Z[G]$, we use $\chi(A)$ to denote the character sum $\sum_{x \in A} \chi(x)$. For a more detailed treatment of group rings and characters, please refer to \cite{Pott95}.

According to Proposition~\ref{prop-translation}, the definition of formally dual pairs can be phrased in the following way:

\begin{definition}\label{def-ori}
Let $S$ be a subset of $G$ and $T$ be a subset of $\wh{G}$. If for each $\chi \in \wh{G}$,
$$
|\chi(S)|^2=\frac{|S|^2}{|T|}\nu_T(\chi),
$$
then $S$ and $T$ form a formally dual pair in $G$ and $\wh{G}$.
\end{definition}

\begin{remark}\label{rem-trans}
By Definition~\ref{def-ori}, the formal duality depends only on the differences generated from $S$ and $T$. Consequently, formal duality is invariant under translation. Let $S$ and $T$ be a formally dual pair in $G$ and $\wh{G}$. Suppose $S^{\pr}$ is a translation of $S$ and $T^{\pr}$ is a translation of $T$. Then $S^{\pr}$ and $T^{\pr}$ also form a formally dual pair in $G$ and $\wh{G}$.
\end{remark}

The following example says that a formally dual pair trivially exists in each pair of groups $G$ and $\wh{G}$.

\begin{example}\label{exam-subgp}
Let $H$ be a subgroup of $G$. Suppose $S=H$ and $T=H^{\perp}$. Then $S$ and $T$ form a formally dual pair in $G$ and $\wh{G}$.
\end{example}

The following proposition suggests that Example~\ref{exam-subgp} is degenerate, in the sense that it can be generated by lifting a formally dual pair in $H$ and $H^{\perp}$.

\begin{proposition}\label{prop-lifting}
Let $G$ be a group and $H \leqslant G$ be a proper subgroup. Let $\phi: \wh{G} \rightarrow \wh{H}$ be a group homomorphism, such that $\phi(\chi)=\chi\mid_H$ for each $\chi \in \wh{G}$, where $\chi\mid_H$ is the restriction of $\chi$ on $H$.
\begin{itemize}
\item[(1)] {\rm \cite[Lemma 4.1]{CKRS}} $S$ and $T$ form a formally dual pair in $H$ and $\wh{H}$ if and only if $S$ and $\phi^{-1}(T)$ form a formally dual pair in $G$ and $\wh{G}$.
\item[(2)] {\rm \cite[Lemma 4.2]{CKRS}} Let $S$ and $T$ be a formally dual pair in $G$ and $\wh{G}$. Suppose $S \subset H \leqslant G$. Then $S$ and $\phi(T)$ form a formally dual pair in $H$ and $\wh{H}$.
\end{itemize}
\end{proposition}

From the perspective of Remark~\ref{rem-trans} and Proposition~\ref{prop-lifting}, a formally dual pair consisting of $S \subset G$ and $T \subset \wh{G}$ is \emph{nondegenerate} if $S$ is not contained in a coset of a proper subgroup of $G$ and $T$ is not contained in a coset of a proper subgroup of $\wh{G}$. Moreover, by Remark~\ref{rem-trans} and Proposition~\ref{prop-lifting}(1), if $S$ is contained in a coset of a proper subgroup $H$ of $\wh{G}$, then $T$ is a union of cosets of a nontrivial subgroup $H^{\perp}$ in $\wh{G}$. Consequently, we have the following concept of primitive subsets in a group.

\begin{definition}[Primitive subset]
For a subset $S$ of a group $G$, the set $S$ is a primitive subset if $S$ is not contained in a coset of a proper subgroup of $G$ and $S$ is not a union of cosets of a nontrivial subgroup in $G$.
\end{definition}

Furthermore, we proceed to introduce the concept of primitive formally dual pairs, which was proposed in \cite[p. 134]{CKRS}.

\begin{definition}[Primitive formally dual pair]\label{def-primi}
Let $S$ and $T$ be a formally dual pair in $G$ and $\wh{G}$. They form a primitive formally dual pair if $S$ is a primitive subset in $G$ and $T$ is a primitive subset in $\wh{G}$.
\end{definition}

For a formally dual pair, $S$ belongs to the group $G$ and $T$ belongs to the character group $\wh{G}$. Note that $G$ and $\wh{G}$ are isomorphic. In order to facilitate our analysis and construction, we adopt and formalize the idea proposed by Cohn et al. \cite[p. 129]{CKRS}, which identifies the two groups $G$ and $\wh{G}$. By doing so, we make both $S$ and $T$ subsets of $G$, which is favorable for our purpose. Next, we provide a rigid description of this identification.

Let $\De$ be a group isomorphism from $G$ to $\wh{G}$, such that
\begin{align*}
\De: G &\rightarrow \wh{G} \\
     y &\mapsto \chi_y
\end{align*}
for some $\chi_y \in \wh{G}$. Using this $\De$, we can identify the set $T \subset \wh{G}$ with its preimage $\De^{-1}(T) \subset G$. Consequently, the definition of formally dual pairs can be rephrased in the following way:

\begin{definition}\label{def-iso}
Let $\De$ be a group isomorphism from $G$ to $\wh{G}$, such that $\De(y)=\chi_y$ for each $y \in G$. Let $S$ and $T$ be subsets of $G$. Then $S$ and $T$ form a formally dual pair in $G$ under the isomorphism $\De$, if for each $y \in G$,
\begin{equation}\label{eqn-def}
|\chi_y(S)|^2=\frac{|S|^2}{|T|}\nu_T(y)
\end{equation}
and
\begin{equation}\label{eqn-def2}
|\chi_y(T)|^2=\frac{|T|^2}{|S|}\nu_S(y).
\end{equation}
Moreover, $S$ and $T$ form a primitive formally dual pair in $G$, if both $S$ and $T$ are primitive subsets of $G$. For $S \subset G$, if there exists a subset $T \subset G$ such that $S$ and $T$ form a (primitive) formally dual pair in $G$, then $S$ is called a (primitive) formally dual set in $G$.
\end{definition}

\begin{remark}\label{rem-def}
\quad
\begin{itemize}
\item[(1)] We note that by Lemma~\ref{lem-pri}(2) below, if one of $S$ and $T$ is a union of cosets of a nontrivial subgroup in $G$, then the other one is contained in a coset of a proper subgroup of $G$. Thus, the fact that neither of $S$ and $T$ is contained in a coset of a proper subgroup of $G$, guarantees that both $S$ and $T$ are primitive subsets in $G$.
\item[(2)] According to \cite[Remark 2.10]{CKRS}, the roles of $S$ and $T$ are interchangeable. Thus, \eqref{eqn-def} holds for each $y \in G$ if and only if \eqref{eqn-def2} holds for each $y \in G$. Moreover, by the Fourier inversion formula (see Proposition~\ref{prop-fourier}), the character values of $SS^{(-1)}$ and $TT^{(-1)}$ uniquely determine $SS^{(-1)}$ and $TT^{(-1)}$, respectively. Together with \eqref{eqn-def} and \eqref{eqn-def2}, we can see that $SS^{(-1)}$ uniquely determines $TT^{(-1)}$ and vice versa. Thus a formally dual set contains much information about its corresponding formally dual pair. In particular, when we study the nonexistence of primitive formally dual pairs, it suffices to consider the corresponding primitive formally dual sets.
\item[(3)] Note that $|\chi_y(S)|^2$ and $|\chi_y(T)|^2$ are both algebraic integers and rational numbers, therefore they must be nonnegative integers.
\end{itemize}
\end{remark}

Note that there are various isomorphisms between $G$ and $\wh{G}$. Below we are going to show that the specific choice of the isomorphism $\De$ in Definition~\ref{def-iso} is not essential.

\begin{proposition}
Let $\De_1, \De_2$ be two isomorphisms from $G$ to $\wh{G}$. Then $S$ and $T$ form a formally dual pair in $G$ under the isomorphism $\De_1$ if and only if $S$ and $\De_2^{-1}(\De_1(T))$ form a formally dual pair in $G$ under the isomorphism $\De_2$.
\end{proposition}
\begin{proof}
For each $y \in G$, we have
$$
\frac{|S|^2}{|\De_2^{-1}(\De_1(T))|}\nu_{\De_2^{-1}(\De_1(T))}(y)=\frac{|S|^2}{|T|}\nu_T(\De_1^{-1}(\De_2(y))).
$$
Hence, $S$ and $T$ form a formally dual pair in $G$ under the isomorphism $\De_{1}$, if and only if for each $y \in G$,
\begin{align*}
\frac{|S|^2}{|\De_2^{-1}(\De_1(T))|}\nu_{\De_2^{-1}(\De_1(T))}(y)&=\frac{|S|^2}{|T|}\nu_T(\De_1^{-1}(\De_2(y)))=|(\De_1(\De_1^{-1}(\De_2(y))))(S)|^2=|(\De_2(y))(S)|^2,
\end{align*}
which is equivalent to the fact that $S$ and $\De_2^{-1}(\De_1(T))$ form a formally dual pair in $G$ under the isomorphism $\De_2$.
\end{proof}

Thanks to the above proposition, when we talk about a formally dual pair $S$ and $T$ in $G$, we always assume that a specific isomorphism $\De$ from $G$ to $\wh{G}$ is chosen, so that \eqref{eqn-def} and \eqref{eqn-def2} hold for each $y \in G$ and each $\chi_y=\De(y)$. From now on, we always regard a formally dual pair as two subsets of a group $G$.

The following are the simplest examples of primitive formally dual pairs.

\begin{example}\label{exam-tri}{\rm (The trivial configuration)}
Let $S=T=\{1\}$ be subsets of the trivial group $G=\{1\}$. Then $S$ and $T$ form a primitive formally dual pair in $G$.
\end{example}

\begin{example}\label{exam-TITO}{\rm (The TITO configuration)}
Let $S=T=\{1,g\}$ be subsets of $G=\Z_4=\{1,g,g^2,g^3\}$. Then $S$ and $T$ form a primitive formally dual pair in $G$. This example is called the TITO configuration in \cite[p. 131]{CKRS}, which stands for ``two-in two-out''.
\end{example}

Indeed, for primitive formally dual pairs in cyclic groups, we have the following main conjecture.
\begin{conjecture}\label{conj-main}{\rm \cite[p. 135]{CKRS}}
There exists no primitive formally dual pair in cyclic groups, except the trivial configuration and the TITO configuration.
\end{conjecture}

Let $N$ be a positive integer and $p$ be a prime divisor of $N$. For a positive integer $a$, we say $p^a \big\| N$ if $p^a \big| N$ and $p^{a+1} \nmid N$. There are some partial results in favor of Conjecture~\ref{conj-main}, which are summarized below.

\begin{proposition}\label{prop-cyclic}
There exists no primitive formally dual pair in the following cyclic groups.
\begin{itemize}
\item[(1)] $\Z_{p^a}$, where $p$ is a prime and $a \ge 1$ {\rm($p=2$ \cite[Section 4.2]{Xia}, $p$ odd \cite[Theorem 1.1]{Sch})}.
\item[(2)] $\Z_{p^aq}$, where $p$ and $q$ are two distinct primes and $a \ge 1$ {\rm\cite[Proposition 7.4]{Mali}}.
\item[(3)] $\Z_{p^2q^2}$, where $p$ and $q$ are two distinct primes {\rm\cite[Proposition 7.7]{Mali}}.
\item[(4)] $\Z_{p^aq^2}$, where $p$ and $q$ are two distinct primes and $a$ is odd {\rm\cite[Proposition 7.5]{Mali}}.
\item[(5)] $\Z_{p^4q^3}$, where $p$ and $q$ are two distinct primes {\rm\cite[Proposition 7.6]{Mali}}.
\item[(6)] $\Z_{p^3q^3}$, where $p$ and $q$ are two distinct primes, except possibly that $p$ and $q$ are simultaneously twin primes and Wieferich pair {\rm\cite[Proposition A.2]{Mali}}.
\item[(7)] $\Z_{p^aq^3}$, where $p$ and $q$ are two distinct primes, $a \ge 5$ and $p,q<10^3$ {\rm\cite[Proposition A.3]{Mali}}.
\item[(8)] $\Z_{p^aq^b}$, where $p$ and $q$ are two distinct primes, $a,b \ge 4$ and $p,q<10^3$ {\rm\cite[Proposition A.4]{Mali}}.
\item[(9)] $\Z_{p^aq^b}$, where $p$ and $q$ are two distinct primes and $a,b \ge 1$, except possibly finite many exceptions for each pair $(p,q)$ {\rm\cite[Theorem 7.3]{Mali}}.
\item[(10)] $\Z_{N}$, where $p \big\| N$ is a prime and $p$ is self-conjugate modulo $N$ {\rm\cite[Theorem 8.3]{Mali}}.
\end{itemize}
\end{proposition}

In contrast to the case of cyclic groups, which seems to contain very few primitive formally dual pairs, some infinite families of primitive formally dual pairs are known in noncyclic groups. We summarize all known constructions of primitive formally dual pairs below.

\begin{proposition}\label{prop-con}
In the following cases, there exists a primitive formally dual pair $S$ and $T$ in $G$.
\begin{itemize}
\item[(1)] $G=\{1\}$, where $S=T=\{1\}$ {\rm(Example~\ref{exam-tri})}.
\item[(2)] $G=\Z_4=\{1,g,g^2,g^3\}$, where $S=T=\{1,g\}$ {\rm(Example~\ref{exam-TITO})}.
\item[(3)] $p$ odd prime, $G=\Z_p \times \Z_p$, where $S=\{(x,x^2) \mid x \in \Z_p\}$ and $T=\{(ax^2,bx) \mid x \in \Z_p\}$, and $a,b$ are nonzero elements of $\Z_p$ {\rm \cite[Theorem 3.2]{CKRS}}.
\item[(4)] $p$ odd prime, $G=\Z_{p^k} \times \Z_{p^k}$ with $k \ge 2$, where $S=\{(x,x^2) \mid x \in \Z_p\}$ and $T=\{(x^2,x) \mid x \in \Z_p\}$ {\rm \cite[Theorem 4.1]{Xia}}.
\end{itemize}
\end{proposition}

A primitive formally dual pair possesses strong symmetry, which implies strong restrictions on their existence. We summarize the known necessary conditions about primitive formally dual pairs. An integer $a$ is called square-free, if $b^2 \big| a$ holds only when $b=1$. For two integers $a$ and $b$, we use $(a,b)$ to denote their greatest common divisor.
\begin{proposition}\label{prop-non}
Let $S$ and $T$ be a formally dual pair in $G$. Then the following holds.
\begin{itemize}
\item[(1)] $|G|=|S|\cdot|T|$ {\rm\cite[Theorem 2.8]{CKRS}}.
\item[(2)] $(|S|,|T|)>1$. In particular, $|G|$ is not square-free {\rm\cite[Theorem 4.8]{Xia}}.
\item[(3)] If $G=\Z_N$ and $g$ is a generator of $\Z_N$, for each $0 \le i \le N-1$, we have $\nu_S(g^i)=\nu_S(g^{(i,N)})$ and $\nu_T(g^i)=\nu_T(g^{(i,N)})$ {\rm\cite[Theorem 3.1]{Sch}}.
\item[(4)] If $G=\Z_N$ and $N$ has exactly two prime divisors, then $\nu_{S}(g), \nu_T(g) \ge 1$, where $g$ is a generator of $\Z_N$ {\rm\cite[Lemma 7.1]{Mali}}.
\end{itemize}
\end{proposition}

Now we proceed to deal with the equivalence problem of formally dual pairs. For this purpose, we first introduce more notation.

Since every abelian group $G$ can be expressed as direct products of cyclic groups, there exists a natural inner product $\lan \cdot ,  \cdot \ran$ in $G$ defined as follows.

\begin{definition}
Let $G=\Z_{n_1}\times\Z_{n_2}\times\cdots\times\Z_{n_t}$ and $n=\exp(G)$. For $x=(x_1,x_2,\ldots,x_t) \in G$ and $y=(y_1,y_2,\ldots,y_t) \in G$, the inner product of $x$ and $y$ is
$$
\lan x,y \ran= \frac{n}{n_1}x_1y_1+\frac{n}{n_2}x_2y_2+\cdots+\frac{n}{n_t}x_ty_t \pmod{n}.
$$
\end{definition}

For $\phi \in \Aut(G)$ and each $y \in G$, there exists a unique element $y^{\pr} \in G$, such that $\lan \phi(x),y \ran=\lan x,y^{\pr} \ran$ holds for every $x \in G$. Thus, $\phi$ induces a bijection $\phi^*$ on $G$, such that $\lan \phi(x),y \ran=\lan x,\phi^*(y) \ran$ holds for each $x,y \in G$. By the properties of the inner product, we can see that $\phi^* \in \Aut(G)$, and $\phi^*$ is called the \emph{adjoint} of $\phi$.

According to Remark~\ref{rem-trans}, formal duality is invariant under translation. Let $S$ and $T$ be a formally dual pair in $G$. The following proposition is a discrete analogue of \cite[Lemma 2.4]{CKRS} and \cite[Lemma 2]{CKS}, which says that the action of $\Aut(G)$ on a formally dual pair produces a series of essentially the same formally dual pairs. Note that for $A \subset G$ and $g \in G$, we use $gA$ to denote the set $\{ga \mid a \in A\}$.

\begin{proposition}\label{prop-auto}
Let $G$ be a group. Let $\phi \in \Aut(G)$ and $\phi^*$ be the adjoint of $\phi$. Suppose $S$ and $T$ form a formally dual pair in $G$. Then $\phi(S)$ and $(\phi^*)^{-1}(T)$ also form a formally dual pair in $G$.
\end{proposition}
\begin{proof}
Let $n=\exp(G)$ and $\zeta_n$ be an $n$-th primitive root of unity. Since $S$ and $T$ form a formally dual pair, for each $y \in G$, we have
$$
\frac{|\phi(S)|^2}{|T|}\nu_{(\phi^*)^{-1}(T)}(y)=\frac{|S|^2}{|T|}\nu_{T}(\phi^*(y))=|\chi_{\phi^{*}(y)}(S)|^2.
$$
Moreover, for each $y \in G$,
$$
|\chi_y(\phi(S))|^2=\sum_{g \in [SS^{(-1)}]} \zn^{\lan \phi(g),y \ran}=\sum_{g \in [SS^{(-1)}]} \zn^{\lan g,\phi^*(y) \ran}=|\chi_{\phi^{*}(y)}(S)|^2=\frac{|\phi(S)|^2}{|T|}\nu_{(\phi^*)^{-1}(T)}(y).
$$
Therefore, $\phi(S)$ and $(\phi^*)^{-1}(T)$ also form a formally dual pair in $G$.
\end{proof}

Now we are ready to introduce the concept of equivalence between formally dual sets.

\begin{definition}\label{def-equiv}
Let $S$ and $S^\pr$ be two formally dual sets in $G$. They are equivalent if there exist $g \in G$ and $\phi \in \Aut(G)$, such that
$$
S^\pr=g\phi(S).
$$
\end{definition}

Let $S$ and $S^{\pr}$ be two equivalent formally dual sets in $G$, such that $S^\pr=g\phi(S)$ for some $g \in G$ and $\phi \in \Aut(G)$. Suppose $S$ and $T$ form a formally dual pair in $G$. Then, by Remark~\ref{rem-trans} and Proposition~\ref{prop-auto}, $S^{\pr}$ and $T^{\pr}=h((\phi^*)^{-1}(T))$ form a formally dual pair in $G$, for each $h \in G$. Thus, $T$ and $T^{\pr}$ are also equivalent. We regard the above two formally dual pairs as equivalent ones. Moreover, the equivalence between these two formally dual pairs can be reduced to the equivalence between two formally dual sets.

For $A \in \Z[G]$, we call the multiset
$$
\{|\chi(A)|^2 \mid \chi \in \wh{G}\}
$$
the \emph{character spectrum} of $A$, and the multiset
$$
\{[AA^{(-1)}]_g \mid g \in G\}
$$
the \emph{difference spectrum} of $A$. Clearly, the character spectrum and difference spectrum are invariants with respect to equivalence. If two formally dual sets have distinct character spectra, or distinct difference spectra, then they are inequivalent.

Finally, we mention the following simple criterion to determine the primitivity of a subset.

\begin{lemma}\label{lem-priset}
$S$ is contained in a coset of a proper subgroup $H$ of $G$ if and only if there exists a nonprincipal character $\chi$, such that $|\chi(S)|^2=|S|^2$.
\end{lemma}
\begin{proof}
Suppose $S$ is contained in a coset of a proper subgroup $H$. Then, by choosing a nonprincipal $\chi$ which is principal on $H$, we have $|\chi(S)|^2=|S|^2$. Conversely, assume that $|\chi(S)|^2=|S|^2$ for a nonprincipal character $\chi$. Thus, $\ker \chi$ is a proper subgroup of $G$. Note that $|\chi(S)|^2=|S|^2$ implies that for each $s_1,s_2 \in S$, we have $s_1s_2^{-1} \in \ker \chi$. Therefore, $S$ is contained in a coset of a proper subgroup $\ker \chi$.
\end{proof}

We also have the following criteria to determine the primitivity of formally dual pairs.

\begin{lemma}\label{lem-pri}
Let $S$ and $T$ be a formally dual pair in $G$.
\begin{itemize}
\item[(1)] If for a nonidentity $y \in G \sm \{1\}$, we have $\nu_S(y)=|S|$ or $\nu_T(y)=|T|$, then $S$ and $T$ do not form a primitive formally dual pair.
\item[(2)] Let $H$ be a nontrivial subgroup of $G$. If one of $S$ and $T$ is a union of cosets of $H$, then the other one is contained in a coset of a proper subgroup of $G$. Thus, $S$ and $T$ do not form a primitive formally dual pair.
\end{itemize}
\end{lemma}
\begin{proof}
(1) Without loss of generality, assume that $\nu_S(y)=|S|$, where $y \in G \sm \{1\}$. By \eqref{eqn-def2}, $|\chi_y(T)|^2=|T|^2$ for a nonprincipal $\chi_y \in \wh{G}$. By Lemma~\ref{lem-priset}, $T$ is not a primitive subset. Therefore $S$ and $T$ do not form a primitive formally dual pair.

(2) Let $\De: G \rightarrow \wh{G}$ be a group isomorphism, such that $\De(y)=\chi_y$ for each $y \in G$. Without loss of generality, assume that $S$ is a union of cosets of a nontrivial subgroup $H$. Let $L=\De^{-1}(H^\perp)$, which is a proper subgroup of $G$. Therefore, for each $y \in G \sm L$, the character $\chi_y$ is nonprincipal on $H$. Note that $S$ is a union of cosets of $H$. For any $y \in G \sm L$, we have $\chi_y(S)=0$, which implies $\nu_T(y)=0$. Thus, for any $t_1, t_2 \in T$, we have $t_1t_2^{-1} \in L$. This means that $T$ is contained in a coset of a proper subgroup $L$ and is therefore not a primitive subset. Thus, $S$ and $T$ do not form a primitive formally dual pair.
\end{proof}

\section{Constructions of primitive formally dual pairs}\label{sec3}

\subsection{Preliminaries}

In this subsection, we introduce the Fourier inversion formula. For a more detailed treatment, please refer to \cite{Pott95}.

Let $\wh{G}$ denote the character group of a group~$G$. Each character $\chi \in \wh{G}$ is extended linearly to the group ring $\Z[G]$. Given $A \in \Z[G]$, the following proposition shows that the character values of $A$ uniquely determine $A$.

\begin{proposition}[Fourier inversion formula] \label{prop-fourier}
Let $G$ be a group and let $A=\sum_{g \in G} a_gg \in \Z[G]$. Then for each $g \in G$, we have
$$
a_g=\frac{1}{|G|}\sum_{\chi \in \wh{G}} \chi(A)\ol{\chi(g)}.
$$
\end{proposition}

Let $\De: G \rightarrow \wh{G}$ be an isomorphism with $\De(y)=\chi_y$ for each $y \in G$. For each subgroup $N$ of $G$, define  $\wti{N}:=\wti{N}(\De)=\De^{-1}(N^\perp)$. Therefore, $\wti{N}=\{y \in G \mid \mbox{$\chi_y$ principal on $N$}\}$. Note that $\wti{N}$ depends on the specific choice of the group isomorphism $\De$. We assume that a proper group isomorphism is chosen, whenever we mention $\wti{N}$ below.

Let $N$ be a normal subgroup of $G$ and $\rho: G \rightarrow G/N$ be the natural projection. Each character $\chi \in \wh{G}$ principal on $N$ induces a character $\wti{\chi} \in \wh{G/N}$, such that $\wti{\chi}(\rho(g))=\chi(g)$ for each $g \in G$. Conversely, each $\wti{\chi} \in \wh{G/N}$ induces a \emph{lifting character} $\chi \in \wh{G}$ satisfying $\chi(g)=\wti{\chi}(\rho(g))$ for every $g \in G$.

Finally, we mention a product construction of formally dual pairs.

\begin{proposition}{\rm (Product construction) \cite[Lemma 3.1]{CKRS}}\label{prop-prod}
Let $S_1$ and $T_1$ be a formally dual pair in $G_1$. Let $S_2$ and $T_2$ be a formally dual pair in $G_2$. Then $S_1 \times S_2$ and $T_1 \times T_2$ form a formally dual pair in $G_1 \times G_2$.
\end{proposition}

As we shall see, this product construction provides a powerful approach to generate inequivalent primitive formally dual pairs in the same group.

\subsection{Construction from relative difference sets}

We observe that the primitive formally dual sets $S$ and $T$ in Proposition~\ref{prop-con}(3), are simply some well-studied configuration called relative difference sets (RDSs). Next, we introduce some basic knowledge about RDSs.  For a more detailed treatment, please refer to \cite{Pott95}.

\begin{definition}\label{def-RDS}
Let $G$ be a group of order $mn$ and $N$ be a subgroup of order $n$. A $k$-subset $R \subset G$ is an $(m,n,k,\la)$ relative difference set in $G$ relative to $N$, if
$$
RR^{(-1)}=k+\la(G-N).
$$
Equivalently, $R \subset G$ is an $(m,n,k,\la)$ relative difference set in $G$, if
$$
|\chi(R)|^2=\begin{cases}
  k^2 & \mbox{if $\chi$ is principal,} \\
  k-\la n & \mbox{if $\chi$ is nonprincipal, but principal on $N$,} \\
  k &\mbox{if $\chi$ is nonprincipal on $N$.}
\end{cases}
$$
\end{definition}

We notice that the primitive formally dual sets $S$ and $T$ in Proposition~\ref{prop-con}(3), are both $(p,p,p,1)$-RDSs in $\Z_p \times \Z_p$. Therefore, we are particularly interested in $(n,n,n,1)$-RDSs, which have been intensively studied.

\begin{proposition}
Let $R$ be an $(n,n,n,1)$-RDS in a group $G$ relative to a subgroup $N$.
\begin{itemize}
\item[(1)] {\rm \cite[Result 5.4.1]{Pott95}} If $n$ is even, then $n=2^m$, the group $G$ is isomorphic to $\Z_4^m$ and $N$ is isomorphic to $2\Z_4^m$.
\item[(2)] {\rm \cite[Theorem 1.1]{BJS}} If $n$ is odd, then $n=p^m$ for some odd prime $p$.
\end{itemize}
\end{proposition}

We observe that when $p$ is odd, all \emph{known} $(p^m,p^m,p^m,1)$-RDSs live in the elementary abelian group $\Zp^{2m}$. Below, we give two examples of $(n,n,n,1)$-RDSs. Note that the definition of Teichm\"uller set can be found at the beginning of the next subsection.

\begin{example}\label{exam-RDScon}
\quad
\begin{itemize}
\item[(1)] {\rm \cite[Theorem 1(b)]{BD}} Let $R$ be the Galois ring $\GR(4,m)$, whose additive group is $\Z_4^m$. Let $\cT$ be the Teichm\"uller set of $R$. Then $\cT$ is a $(2^m,2^m,2^m,1)$-RDS in $\Z_4^m$ relative to $2\Z_4^m$.
\item[(2)] {\rm \cite[Theorem 3.13]{Pott16}} Let $H_1$ and $H_2$ be two groups. A function $f:H_1 \rightarrow H_2$ is \emph{planar}, if the mapping $f(ax)f(x)^{-1}$ is bijective for each nonidentity $a$. Let $p$ be an odd prime. Let $G$ be the elementary abelian group $\Z_p^{2m}$ and $H$ be a subgroup of $G$ isomorphic to $\Z_p^m$. Let $f$ be a planar function from $H$ to $H$. Then, $\{(x,f(x)) \mid x \in \Z_{p^m}\}$ is a $(p^m,p^m,p^m,1)$-RDS in $G=H \times H$ relative to $\{0\} \times H$. We regard the elementary abelian group $\Z_p^m$ as the additive group of the finite field $\Fpm$, the multiplication among elements of $\Z_p^m$ is well-defined. Note that $f(x)=x^2$ is a planar function from $\Z_p^m$ to $\Z_p^m$. Hence, $\{(x,x^2) \mid x \in \Z_p^m\}$ is a $(p^m,p^m,p^m,1)$-RDS in $\Z_p^m \times \Z_p^m$ relative to $\{0\}\times \Z_p^m$.
\end{itemize}
\end{example}

\begin{remark}
\quad
\begin{itemize}
\item[(1)] Example~\ref{exam-RDScon}(1) has been generalized in \cite{Zhou} using planar functions over characteristic $2$.
\item[(2)] The above examples can be rephrased using the language of presemifields (see \cite[Construction 8.8]{Pott16}). There are many known presemifields which produce numerous $(n,n,n,1)$-RDSs \cite{Kan}, \cite[Sections 8.2, 8.3]{Pott16}.
\end{itemize}
\end{remark}

Let $\De: G \rightarrow \wh{G}$ be an isomorphism with $\De(y)=\chi_y$ for each $y \in G$. Recall that for each subgroup $N$ of $G$, the subgroup $\wti{N}=\{y \in G \mid \mbox{$\chi_y$ principal on $N$}\}$. Thus, we have $y=1$ if and only if $\chi_y$ is principal on $G$, the element $y \in \wti{N} \setminus \{1\}$ if and only if $\chi_y$ is nonprincipal, but principal on $N$, the element $y \in G \sm \wti{N}$ if and only if $\chi_y$ is not principal on $N$. Now, we are ready to extend the construction in Proposition~\ref{prop-con}(3) from the viewpoint of RDSs.

\begin{theorem}\label{thm-RDS}
Let $S$ be an $(n,n,n,1)$-RDS in $G$ relative to $N$. Then $S$ and $T$ form a primitive formally dual pair in $G$, if and only if \,$T$ is an $(n,n,n,1)$-RDS in $G$ relative to $\wti{N}$. More concretely,
\begin{itemize}
\item[(1)] When $n$ is even, let $S$ and $T$ be two $(2^m,2^m,2^m,1)$-RDSs in $\Z_4^m$ relative to $2\Z_4^m$. Then $S$ and $T$ form a primitive formally dual pair in $\Z_4^m$. In particular, let $R$ be the Galois ring $\GR(4,m)$ and $\cT$ be the Teichm\"uller set of $R$. Then $S=\cT$ and $T=\cT$ form a primitive formally dual pair in $\Z_4^m$.
\item[(2)] When $n$ is odd, $H=\Z_p^m$. Let $f_1$ and $f_2$ be two planar functions from $H$ to $H$. Then, $S=\{(x,f_1(x)) \mid x \in H\}$ and $T=\{(f_2(x),x) \mid x \in H\}$  form a primitive formally dual pair in $\Z_p^{2m}$.
\end{itemize}
\end{theorem}
\begin{proof}
Since $S$ is an $(n,n,n,1)$-RDS in $G$ relative to $N$, by definition,
$$
|\chi_y(S)|^2=\begin{cases}
  n^2 & \mbox{if $y=0$,} \\
  0  & \mbox{if $y \in \wti{N} \sm \{0\}$,} \\
  n  & \mbox{if $y \in G \sm \wti{N}$.}
\end{cases}
$$
By Proposition~\ref{prop-non}(1), we have $|T|=n$. By Definition~\ref{def-iso}, $S$ and $T$ form a formally dual pair if and only if
$$
\nu_T(y)=\begin{cases}
  n & \mbox{if $y=0$,} \\
  0  & \mbox{if $y \in \wti{N} \sm \{0\}$,} \\
  1  & \mbox{if $y \in G \sm \wti{N}$,}
\end{cases}
$$
which means that $T$ is an $(n,n,n,1)$-RDS in $G$ relative to $\wti{N}$. Since both $S$ and $T$ are RDSs, each of them cannot be contained in a coset of a proper subgroup of $G$. Thus, $S$ and $T$ form a primitive formally dual pair and therefore, the necessary and sufficient condition is proved.

When $n$ is even, we choose $G=\Z_4^m$. Let $S$ be a $(2^m,2^m,2^m,1)$-RDS in $G$ relative to $N=2\Z_4^{m}$. For $y \in \Z_4^m$, define a character $\chi_y(a)=(\sqrt{-1})^{y\cdot a}$ for each $a \in \Z_4^m$, where $\cdot$ represents the usual inner product in $\Z_4^m$. Let $\De: y \mapsto \chi_y$ be a group isomorphism from $G$ to $\wh{G}$. Then, we have $\wti{N}=\wti{N}(\De)=2\Z_4^m$. Note that $T$ is a $(2^m,2^m,2^m,1)$-RDS in $G$ relative to $\wti{N}=2\Z_4^{m}$. Then, $S$ and $T$ form a primitive formally dual pair in $\Z_4^m$. In particular, by Example~\ref{exam-RDScon}(1), we can choose $S=T=\cT$.

When $n$ is odd, we choose $G=\Z_p^{2m}$, such that $G=H \times H$. By Example~\ref{exam-RDScon}(2), $S$ is a $(p^m,p^m,p^m,1)$-RDS in $G$ relative to $N=\{0\} \times H$. Each $y \in \Z_p^{2m}$ can be written as $y=(y_1,y_2) \in \Z_p^m \times \Z_p^m$. For $y=(y_1,y_2) \in \Z_p^m \times \Z_p^m$, define a character $\chi_{y_1,y_2}(a_1,a_2)=\zp^{y_1 \cdot a_1+y_2 \cdot a_2}$ for each $(a_1,a_2) \in \Z_p^m \times \Z_p^m$, where $\zp$ is a primitive $p$-th root of unity and $\cdot$ represents the usual inner product in $\Z_p^m$. Let $\De^{\pr}: (y_1,y_2) \mapsto \chi_{y_1,y_2}$ be a group isomorphism from $G$ to $\wh{G}$. Then, we have $\wti{N}=\wti{N}(\De^{\pr})=H \times \{0\}$. Note that by Example~\ref{exam-RDScon}(2), $T$ is a $(p^m,p^m,p^m,1)$-RDS in $G$ relative to $\wti{N}=H \times \{0\}$. Then $S$ and $T$ form a primitive formally dual pair in $\Z_p^{2m}$.
\end{proof}

\begin{remark}
\quad
\begin{itemize}
\item[(1)] Let $p$ be an odd prime. By Example~\ref{exam-RDScon}(2), $f(x)=x^2$ is a planar function from $\Z_p$ to $\Z_p$. Applying Theorem~\ref{thm-RDS}(2), we recover the primitive formally dual pairs in Proposition~\ref{prop-con}(3).
\item[(2)] Theorem~\ref{thm-RDS} indicates that given an $(n,n,n,1)$-RDS $S$ in $G$ relative to $N$, every $(n,n,n,1)$-RDS $T$ in $G$ relative to $\wti{N}$ can be used to form a primitive formally dual pair in $G$. We note there are many inequivalent $(n,n,n,1)$-RDS (see \cite{Kan} and \cite[Sections 8.2, 8.3]{Pott16}). Consequently, given a primitive formally dual set $S$, there can be more than one primitive formally dual sets $T$, which are pairwise inequivalent, so that $S$ and $T$ form a primitive formally dual pair in $G$.
\end{itemize}
\end{remark}

The following two examples show that by applying the product construction to $(n,n,n,1)$-RDSs and products of the TITO configurations, we can obtain many inequivalent primitive formally dual pairs in one group.

\begin{example}
Let $m \ge 2$. We aim to show that there are at least $m+1$ inequivalent primitive formally dual pairs in $G=\Z_4^m$. Write $m=m_1+m_2$, where $0 \le m_1, m_2 \le m$. Let $S_1=T_1$ be a product of $m_1$ copies of the TITO configuration in $\Z_4^{m_1}$ and $S_2=T_2$ be a $(2^{m_2},2^{m_2},2^{m_2},1)$-RDS in $\Z_4^{m_2}$. Then by the product construction, $S_1 \times S_2$ and $T_1 \times T_2$ form a formally dual pair in $\Z_4^m=\Z_4^{m_1} \times \Z_4^{m_2}$, such that
$$
\{|\chi(S_1 \times S_2)|^2 \mid \chi \in \wh{G}\}=\{|\chi(T_1 \times T_2)|^2 \mid \chi \in \wh{G}\}=\{0,2^{m+s},2^{m+m_2+s} \mid 0 \le s \le m_1\}.
$$
Since the multiplicity of $4^m$ is $1$, by Lemma~\ref{lem-priset}, $S_1 \times S_2$ and $T_1 \times T_2$ form a primitive formally dual pair in $G$. Moreover, since the multiplicity of $0$ is equal to $4^m-3^{m_1}(4^{m_2}-2^{m_2}+1)$, distinct choices of $0 \le m_1 \le m$ yield distinct character spectra of $S_1 \times S_2$ and $T_1 \times T_2$. Thus, there are at least $m+1$ inequivalent primitive formally dual pairs in $\Z_4^m$.
\end{example}

\begin{example}
Let $p$ be an odd prime and $m \ge 2$. We aim to show that there are at least $\lf \frac{m}{2} \rf+1$ inequivalent primitive formally dual pairs in $G=\Z_p^{2m}$. Write $m=m_1+m_2$, where $0 \le m_1 \le m_2 \le m$. Let $S_1=T_1$ be a $(p^{m_1},p^{m_1},p^{m_1},1)$-RDS in $\Z_p^{2m_1}$ and $S_2=T_2$ be a $(p^{m_2},p^{m_2},p^{m_2},1)$-RDS in $\Z_p^{2m_2}$. Then by the product construction, $S_1 \times S_2$ and $T_1 \times T_2$ form a formally dual pair in $\Z_p^{2m}=\Z_p^{2m_1} \times \Z_p^{2m_2}$, such that
$$
\{|\chi(S_1 \times S_2)|^2 \mid \chi \in \wh{G}\}=\{|\chi(T_1 \times T_2)|^2 \mid \chi \in \wh{G}\}=\{0,p^{m},p^{m+m_1},p^{m+m_2},p^{2m}\}.
$$
Since the multiplicity of $p^{2m}$ is equal to $1$, by Lemma~\ref{lem-priset}, $S_1 \times S_2$ and $T_1 \times T_2$ form a primitive formally dual pair in $G$. Moreover, distinct choices of $0 \le m_1 \le \lf \frac{m}{2} \rf$ yield distinct character spectra of $S_1 \times S_2$ and $T_1 \times T_2$. Thus, there are at least $\lf \frac{m}{2} \rf+1$ inequivalent primitive formally dual pairs in $\Zp^{2m}$.
\end{example}

\subsection{Construction from generalized relative difference sets}

In this subsection, we aim to extend the construction in Proposition \ref{prop-con}(4). Although the primitive formally dual sets in that construction are not RDSs any more, we realize they are indeed some natural extensions of RDSs. Next, we introduce the concept of generalized relative difference sets (GRDSs), which live in the products of Galois rings. For this purpose, we first recall some basic facts about Galois rings. For a more detailed treatment, please refer to \cite[Chapter 14]{Wan}.

Let $p$ be a prime and $t$ be a positive integer. Let $f \in \Z_{p^t}[x]$ be a monic polynomial of degree $s$ such that the image of $f$ under the natural projection $\Z_{p^t}[x] \rightarrow \Zp[x]$ is irreducible over $\Zp$. Then the quotient ring $\Z_{p^t}[x]/(f(x))$ is called a Galois ring of characteristic $p^t$ and rank $s$, denoted by $\GR(p^t,s)$. For $0 \le i \le t$, define $(p^i)=p^i\GR(p^t,s)$. The Galois ring $\GR(p^t,s)$ contains a chain of principal ideals $\{0\}=(p^t) \subset (p^{t-1}) \subset \cdots \subset (p) \subset \GR(p^t,s)$. In this section, we always use $R$ to denote the Galois ring $\GR(p^t,s)$. For $0 \le i \le t-1$, define $R_i=(p^i) \sm (p^{i+1})$, and $R_t=\{0\}$. Then $\{R_i \mid 0 \le i \le t\}$ forms a partition of $R$. For each $a \in R$, define $v_p(a)=i$, if $a \in R_i$, where $0 \le i \le t$.

The group of units $R_0$ contains a unique cyclic subgroup of order $p^s-1$, which is denoted by $\cT^*$. The Teichm\"uller set of $R$ is $\cT=\cT^* \cup \{0\}$. Each $x \in R$ can be uniquely expressed as $x=\sum_{i=0}^{t-1} p^ix_i$, where $x_i \in \cT$ for each $0 \le i \le t-1$. The Frobenius automorphism $\sig$ of the Galois ring $R$ is defined to be
\begin{align*}
\sig: R &\rightarrow R, \\
        \sum_{i=0}^{t-1} p^i x_i &\mapsto \sum_{i=0}^{t-1} p^i x_i^p,
\end{align*}
where $x_i \in \cT$ for each $0 \le i \le t-1$. For each $x \in R$, the trace function from $R$ to $\Z_{p^t}$ is defined as
$$
\Tr(x)=\sum_{i=0}^{s-1} \sig^{i}(x).
$$

For each $a \in R$, we can define a character $\chi_a$ from the additive group of $R$ to the multiplicative group of the complex field $\C$, such that $\chi_a(x)=\zpt^{\Tr(ax)}$ for each $x \in R$, where $\zpt$ is a primitive $p^t$-th root of unity. Moreover, the set $\{\chi_a \mid a \in R \}$ contains all characters of $R$. For $0 \le i \le t$, it is easy to see that
\begin{equation}\label{eqn-GR1}
\chi_a((p^i))=\begin{cases}
  p^{(t-i)s} & \mbox{if $v_p(a) \ge t-i$,} \\
  0          & \mbox{if $v_p(a) \le t-i-1$.}
\end{cases}
\end{equation}
Consequently, for $0 \le i \le t-1$,
\begin{equation}\label{eqn-GR2}
\chi_a(R_i)=\begin{cases}
  p^{(t-i)s}-p^{(t-i-1)s} & \mbox{if $v_p(a) \ge t-i$,} \\
  -p^{(t-i-1)s} & \mbox{if $v_p(a)=t-i-1$,} \\
  0          & \mbox{if $v_p(a) < t-i-1$,}
\end{cases}
\end{equation}
and $\chi_a(R_t)=1$.

Now we are ready to define the generalized relative difference sets.

\begin{definition}\label{def-GRDS}
Let $R=\GR(p^t,s)$ and $G=R\times R$. A set $S \subset G$ is a generalized relative difference set in $G$ relative to $\{R_i \times (p^i) \mid 0 \le i \le t\}$, if
$$
\nu_{S}((a,b))=\begin{cases}
  p^{v_p(a)s} & \mbox{if $v_p(a) \le v_p(b)$,} \\
  0           & \mbox{if $v_p(a) > v_p(b)$.}
\end{cases}
$$
Equivalently,
$$
SS^{(-1)}=\sum_{i=0}^t p^{is} (R_i \times (p^i)).
$$
A set $T \subset G$ is a generalized relative difference set in $G$ relative to $\{(p^i) \times R_i \mid 0 \le i \le t\}$, if
$$
\nu_{T}((a,b))=\begin{cases}
  p^{v_p(b)s} & \mbox{if $v_p(a) \ge v_p(b)$,} \\
  0           & \mbox{if $v_p(a) < v_p(b)$.}
\end{cases}
$$
Equivalently,
$$
TT^{(-1)}=\sum_{i=0}^t p^{is} ((p^i) \times R_i).
$$
\end{definition}

Let $R=\GR(p^t,s)$ and $G=R \times R$. Below, we also use $R$ and $R \times R$ to represent their additive groups. For $(a,b) \in R \times R$, we can define a character $\chi_{a,b}$ from $R \times R$ to the multiplicative group of the complex field $\C$, such that $\chi_{a,b}((x,y))=\zpt^{\Tr(ax+by)}$ for each $(x,y) \in R \times R$, where $\zpt$ is a primitive $p^t$-th root of unity. Indeed, $\{\chi_{a,b} \mid a,b \in R\}$ is the set of all characters of $R \times R$.  Define a group isomorphism $\De: R \times R \rightarrow \wh{R \times R}$, which satisfies $\De((a,b))=\chi_{a,b}$ for each $(a,b) \in R \times R$. Throughout the rest of this subsection, we always implicitly use this isomorphism in our construction. The following proposition describes GRDSs from the viewpoint of their character values.

\begin{proposition}\label{prop-GRDSchar}
Let $R=\GR(p^t,s)$ and $G=R\times R$. A set $S \subset G$ is a GRDS in $G$ relative to $\{R_i \times (p^i) \mid 0 \le i \le t\}$, if and only if
$$
|\chi_{a,b}(S)|^2=\begin{cases}
  p^{(t+v_p(b))s} & \mbox{if $v_p(a) \ge v_p(b)$,} \\
  0               & \mbox{if $v_p(a) < v_p(b)$.}
\end{cases}
$$
A set $T \subset G$ is a GRDS in $G$ relative to $\{(p^i) \times R_i \mid 0 \le i \le t\}$, if and only if
$$
|\chi_{a,b}(T)|^2=\begin{cases}
  p^{(t+v_p(a))s} & \mbox{if $v_p(a) \le v_p(b)$,} \\
  0               & \mbox{if $v_p(a) > v_p(b)$.}
\end{cases}
$$
\end{proposition}
\begin{proof}
It suffices to prove the first equation since the second one is analogous. By Definition~\ref{def-GRDS},
$$
SS^{(-1)}=\sum_{i=0}^t p^{is} (R_i \times (p^i)).
$$
Therefore,
$$
|\chi_{a,b}(S)|^2=\sum_{i=0}^t p^{is}\chi_a(R_i)\chi_b((p^i))=p^{ts}+\sum_{i=0}^{t-1} p^{is} \chi_a(R_i)\chi_b((p^i)).
$$
To compute $|\chi_{a,b}(S)|^2$, we are going to use \eqref{eqn-GR1} and \eqref{eqn-GR2} frequently below.

If $v_p(a) \ge v_p(b)$, then
\begin{align*}
|\chi_{a,b}(S)|^2&=p^{ts}+\sum_{i=t-v_p(b)}^{t-1} p^{is} \chi_a(R_i)\chi_b((p^i)) \\
                 &=p^{ts}+\sum_{i=t-v_p(b)}^{t-1} p^{is} (p^{(t-i)s}-p^{(t-i-1)s}) p^{(t-i)s} \\
                 &=p^{(t+v_p(b))s}.
\end{align*}

If $v_p(a)<v_p(b)$, then
\begin{align*}
|\chi_{a,b}(S)|^2&=p^{ts}+\sum_{i=t-v_p(a)-1}^{t-1} p^{is} \chi_a(R_i)\chi_b((p^i)) \\
                 &=p^{ts}+p^{(t-v_p(a)-1)s}(-p^{v_p(a)s})p^{(v_p(a)+1)s}+\sum_{i=t-v_p(a)}^{t-1} p^{is}(p^{(t-i)s}-p^{(t-i-1)s})p^{(t-i)s} \\
                 &=p^{ts}+p^{ts}(-p^{v_p(a)s})+p^{ts}(p^{v_p(a)s}-1) \\
                 &=0.
\end{align*}
Therefore,
$$
|\chi_{a,b}(S)|^2=\begin{cases}
  p^{(t+v_p(b))s} & \mbox{if $v_p(a) \ge v_p(b)$,} \\
  0               & \mbox{if $v_p(a) < v_p(b)$.}
\end{cases}
$$
\end{proof}

Now we are ready to extend the construction in Proposition~\ref{prop-con}(4) from the viewpoint of GRDSs.

\begin{theorem}\label{thm-GRDS}
Let $R=\GR(p^t,s)$ and $G=R\times R$. Let $S$ be a GRDS in $G$ relative to $\{R_i \times (p^i) \mid 0 \le i \le t\}$. Then $S$ and $T$ form a primitive formally dual pair in $G$ if and only if $T$ is a GRDS in $G$ relative to $\{(p^i) \times R_i \mid 0 \le i \le t\}$. Moreover, $T$ can be chosen as $T=\{(x,y) \in G \mid (y,x) \in S\}$.
\end{theorem}
\begin{proof}
By Proposition~\ref{prop-GRDSchar}, we have
$$
|\chi_{a,b}(S)|^2=\begin{cases}
  p^{(t+v_p(b))s} & \mbox{if $v_p(a) \ge v_p(b)$,} \\
  0               & \mbox{if $v_p(a) < v_p(b)$.}
\end{cases}
$$
By Definition~\ref{def-iso}, $S$ and $T$ form a formally dual pair in $G$ if and only if
\begin{equation}\label{eqn-T}
\nu_T((a,b))=\begin{cases}
  p^{v_p(b)s} & \mbox{if $v_p(a) \ge v_p(b)$,} \\
  0           & \mbox{if $v_p(a) < v_p(b)$.}
\end{cases}
\end{equation}
which, by Definition~\ref{def-GRDS}, says that $T$ is a GRDS in $G$ relative to $\{(p^i) \times R_i \mid 0 \le i \le t\}$. Since both $S$ and $T$ are GRDSs, each of them cannot be contained in a coset of a proper subgroup of $G$. Thus, $S$ and $T$ form a primitive formally dual pair. Note that $S$ satisfies
\begin{equation}\label{eqn-S}
\nu_{S}((a,b))=\begin{cases}
  p^{v_p(a)s} & \mbox{if $v_p(a) \le v_p(b)$,} \\
  0           & \mbox{if $v_p(a) > v_p(b)$.}
\end{cases}
\end{equation}
Comparing \eqref{eqn-T} and \eqref{eqn-S}, we can choose $T$ as $T=\{(x,y) \in G \mid (y,x) \in S\}$.
\end{proof}

Let $p$ be an odd prime and $R=\GR(p^t,s)$. The following proposition shows that there exists a GRDS in $R \times R$.
\begin{proposition}\label{prop-square}
Let $p$ be an odd prime. Let $R=\GR(p^t,s)$ be a Galois ring and $G=R \times R$. Then $S=\{(x,x^2) \mid x \in R\}$ is a GRDS in $G$ relative to $\{R_i \times (p^i) \mid 0 \le i \le t\}$.
\end{proposition}
\begin{proof}
For $(a,b) \in G$,
\begin{align*}
|\chi_{a,b}(S)|^2&=\sum_{x \in R} \zpt^{\Tr(ax+bx^2)}\sum_{y \in R} \zpt^{-\Tr(ay+by^2)} \\
                 &=\sum_{y,z \in R} \zpt^{\Tr(a(y+z)+b(y+z)^2-ay-by^2)} \\
                 &=\sum_{z \in R} \zpt^{\Tr(az+bz^2)}\sum_{y \in R} \zpt^{\Tr(2byz)}.
\end{align*}

If $b \ne 0$, then
$$
|\chi_{a,b}(S)|^2=p^{ts} \sum_{z \in (p^{t-v_p(b)})} \zpt^{\Tr(az+bz^2)}.
$$
Note that $bz^2=0$ for each $z \in (p^{t-v_p(b)})$. Therefore,
\begin{align*}
|\chi_{a,b}(S)|^2&=p^{ts} \sum_{z \in (p^{t-v_p(b)})} \zpt^{\Tr(az)} \\
                 &=\begin{cases}
                   p^{(t+v_p(b))s} & \mbox{if $v_p(a) \ge v_p(b)$,} \\
                   0               & \mbox{if $v_p(a) < v_p(b)$.}
                 \end{cases}
\end{align*}
If $b=0$, i.e., $v_p(b)=t$, then
$$
|\chi_{a,b}(S)|^2=p^{ts} \sum_{z \in R} \zpt^{\Tr(az)}=\begin{cases}
                   p^{2ts} & \mbox{if $v_p(a)=t$,} \\
                   0               & \mbox{if $v_p(a) < t$.}
                 \end{cases}
$$
Consequently,
$$
|\chi_{a,b}(S)|^2=\begin{cases}
                   p^{(t+v_p(b))s} & \mbox{if $v_p(a) \ge v_p(b)$,} \\
                   0               & \mbox{if $v_p(a) < v_p(b)$.}
                 \end{cases}
$$
By Proposition~\ref{prop-GRDSchar}, $S$ is a GRDS in $G$ relative to $\{R_i \times (p^i) \mid 0 \le i \le t\}$.
\end{proof}

Combining Theorem~\ref{thm-GRDS} and Proposition~\ref{prop-square}, we have the following corollary.

\begin{corollary}\label{cor-GRDS}
Let $p$ be an odd prime. Let $R=\GR(p^t,s)$ be a Galois ring and $G=R \times R$. Let $S=\{(x,x^2) \mid x \in R\}$ and $T=\{(x^2,x) \mid x \in R\}$. Then $S$ and $T$ form a primitive formally dual pair in $G$.
\end{corollary}

\begin{remark}
\quad
\begin{itemize}
\item[(1)] Applying Corollary~\ref{cor-GRDS} with $s=1$, we reproduce the primitive formally dual pairs of Proposition~\ref{prop-con}(4), which live in $\Z_{p^t} \times \Z_{p^t}$.
\item[(2)] In Proposition~\ref{prop-square}, if $p=2$, then $S$ is not a GRDS in $G$.
\end{itemize}
\end{remark}

\subsection{Construction from skew Hadamard difference sets}

The following is a construction related to skew Hadamard difference sets. We first give a brief account about skew Hadamard difference sets. For a more detailed treatment, please refer to \cite[Chapter VI, Section 8]{BJL}.

\begin{definition}\label{def-DS}
Let $G$ be a group of order $v$. A $k$-subset $D$ of $G$ is a $(v,k,\lambda)$ difference set in $G$, if
$$
DD^{(-1)}=k+\la(G-1).
$$
A difference set $D$ in a group $G$ is a skew Hadamard difference set, if $1$, $D$, $D^{(-1)}$ form a partition of $G$.
\end{definition}

We mention some known results about skew Hadamard difference sets as follows.

\begin{proposition}{\rm\cite[Chapter VI, Theorem 8.7]{BJL}, \cite[Lemma 2.5, Corollary 2.7]{WH}} \label{prop-SHDS}
Let $D$ be a skew Hadamard difference set in an abelian group $G$ of order $v$. Then $D$ is a $(v,\frac{v-1}{2},\frac{v-3}{4})$ difference set in $G$, where $v$ is a prime power congruent to $3$ modulo $4$. Moreover, $D^{(-1)}$ is also a skew Hadamard difference set in $G$. For each nonprincipal character $\chi$ of $G$, we have
$$
(\chi(D),\chi(D^{(-1)})) \in \left\{\left(\frac{-1+\sqrt{-v}}{2},\frac{-1-\sqrt{-v}}{2}\right), \left(\frac{-1-\sqrt{-v}}{2},\frac{-1+\sqrt{-v}}{2}\right)\right\}.
$$
Moreover, $\big\{\chi \in \wh{G} \mid \chi(D)=\frac{-1+\sqrt{-v}}{2}\big\}$ is a skew Hadamard difference set in the character group $\wh{G}$.
\end{proposition}

There are many known constructions of skew Hadamard difference sets in the elementary abelian group $\Zp^m$, where $p$ is an odd prime and $p^m \equiv 3 \pmod 4$. For each $a \in \Zp^m$, define $\chi_a(x)=\zp^{a\cdot x}$ for each $x \in \Zp^m$, where $a\cdot x$ is the usual inner product in $\Zp^m$ and $\zp$ is the $p$-th root of unity. Then, $\{\chi_a \mid a \in \Zp^m\}$ is the set of all characters of $\Zp^m$. Define the {\it dual set} of $D$ as
$$
D^*=\left\{ a \in \Zp^m \mid \chi_a(D)=\frac{-1+\sqrt{-v}}{2} \right\}.
$$
By Proposition~\ref{prop-SHDS}, the dual set $D^*$ is a skew Hadamard difference set in $\Zp^m$.

Based on the definition of skew Hadamard difference sets $D$ and $D^*$, the following lemma follows easily:

\begin{lemma}\label{lem-SHDSchar}
Let $D$ be a skew Hadamard difference set in $\Zp^m$. Set $v=p^m$. For each nonzero $a,b \in \Zp^m$, we have
$$
\setlength{\jot}{5pt}
\chi_a(D)\chi_b(D)=
\left\{\begin{aligned}
  &\frac{1-2\sqrt{-v}-v}{4} & \mbox{if $a,b \in D^*$,} \\
  &\frac{v+1}{4}            & \mbox{if $a \in D^*, b \in D^{*(-1)}$,} \\
  &\frac{1+2\sqrt{-v}-v}{4} & \mbox{if $a,b \in D^{*(-1)}$.}
\end{aligned}\right.
$$
\end{lemma}

For $a \in \Zp^m$ and $\al \in \Zp$, we have $\chi_{a}(\al x)=\chi_{\al a}(x)$ for each $x \in \Zp^m$. Moreover, for $a,b \in \Zp^m$, we have $\chi_{a+b}(x)=\chi_a(x)\chi_b(x)$ for each $x \in \Zp^m$. For $a,b \in \Zp^m$, we know that $\chi_{a,b}$ is defined to be a character of $\Zp^m \times \Zp^m$, where $\chi_{a,b}(x,y)=\chi_a(x)\chi_b(y)$ for each $x,y \in \Zp^m$. Define a group isomorphism $\De: \Zp^m \times \Zp^m \rightarrow \wh{\Zp^m \times \Zp^m}$, which satisfies $\De((a,b))=\chi_{a,b}$ for each $(a,b) \in \Zp^m \times \Zp^m$. Throughout the rest of this subsection, we always implicitly use this isomorphism in our construction. Now, we are ready to describe our construction.

\begin{theorem}\label{thm-SHDS}
Let $p^m \equiv 3 \pmod{4}$ be a prime power. Let $D$ be a skew Hadamard difference set in $\Zp^m$ and $G=\Zp^m \times \Zp^m$. Suppose $\al$ and $\be$ are nonzero elements of $\Zp$ with $\al \ne \be$. Define
$$
S=(0,0)+\sum_{x \in D} (x,\al x)+\sum_{x \in D^{(-1)}} (x, \be x)
$$
and
$$
T=(0,0)+\sum_{x \in D^*} (\frac{\al}{\al-\be}x,\frac{1}{\be-\al} x)+\sum_{x \in D^{*(-1)}} (\frac{\be}{\al-\be}x, \frac{1}{\be-\al} x).
$$
Then $S$ and $T$ form a primitive formally dual pair in $G$.
\end{theorem}
\begin{proof}
Let ${\Zp^m}^*$ be the set of nonzero elements in $\Zp^m$ and $q=p^m$. Since $D$ is a skew Hadamard difference set and $D^*$ is the dual set of $D$, by Proposition~\ref{prop-SHDS},
$$
DD^{(-1)}=D^*D^{*(-1)}=\frac{q-1}{2} \cdot 0+\frac{q-3}{4}{\Zp^m}^*.
$$
Employing the above equation, it is routine to verify that
\begin{align*}
SS^{(-1)}=&q(0,0)+\frac{q+1}{4}\sum_{x \ne 0} (x,\al x)+\frac{q+1}{4}\sum_{x \ne 0} (x,\be x)+\sum_{\substack{x \in D \\ y \in D^{(-1)}}} (x-y,\al x-\be y) \\
        &+\sum_{\substack{x \in D \\ y \in D^{(-1)}}} (y-x,\be y-\al x)
\end{align*}
and
\begin{align*}
TT^{(-1)}=&q(0,0)+\frac{q+1}{4}\sum_{x \ne 0} (\frac{\al}{\al-\be}x,\frac{1}{\be-\al} x)+\frac{q+1}{4}\sum_{x \ne 0} (\frac{\be}{\al-\be}x,\frac{1}{\be-\al} x) \\
        &+\sum_{\substack{x \in D^* \\ y \in D^{*(-1)}}} (\frac{\al x-\be y}{\al-\be},\frac{x-y}{\be-\al})+\sum_{\substack{x \in D^* \\ y \in D^{*(-1)}}} (\frac{\be y-\al x}{\al-\be},\frac{y-x}{\be-\al}).
\end{align*}
Consequently,
$$
\nu_T((a,b))=\begin{cases}
  q & \mbox{if $a=b=0$,} \\
  \frac{q+1}{4} & \mbox{if $a+\al b=0$ or $a+\be b=0$,} \\
  1 & \mbox{if $a+\al b \in D^*,a+\be b \in D^{*(-1)}$ or $a+\al b \in D^{*(-1)},a+\be b \in D^{*}$,} \\
  0 & \mbox{if $a+\al b \in D^*,a+\be b \in D^{*}$ or $a+\al b \in D^{*(-1)},a+\be b \in D^{*(-1)}$.}
\end{cases}
$$

On the other hand, for each $a,b \in \Zp^m$, recall that $\chi_{a+b}(x)=\chi_a(x)\chi_b(x)$ for each $x \in \Zp^m$. Then, we have
\begin{align*}
|\chi_{a,b}(S)|^2=&q+\frac{q+1}{4} \chi_{a+\al b}({\Zp^m}^*)+\frac{q+1}{4} \chi_{a+\be b}({\Zp^m}^*)\\ &+\chi_{a+\al b}(D)\chi_{a+\be b}(D)+\chi_{-(a+\al b)}(D)\chi_{-(a+\be b)}(D).
\end{align*}
Below, we are going to use Lemma~\ref{lem-SHDSchar} frequently.

If $a+\al b=0$, then $a+\be b \ne 0$ and
$$
|\chi_{a,b}(S)|^2=q+\frac{q+1}{4}(q-1)+\frac{q+1}{4}(-1)+\frac{q-1}{2}\chi_{a+\be b}({\Zp^m}^*)=\frac{q^2+q}{4}.
$$
Similarly, if $a+\be b=0$, we also have
$$
|\chi_{a,b}(S)|^2=\frac{q^2+q}{4}.
$$
If $a+\al b \in D^*$, $a+\be b \in D^{*(-1)}$ or $a+\al b \in D^{*(-1)}$, $a+\be b \in D^{*}$, then
$$
|\chi_{a,b}(S)|^2=q-\frac{q+1}{4}-\frac{q+1}{4}+2\frac{(q+1)}{4}=q.
$$
If $a+\al b,a+\be b \in D^*$, or $a+\al b,a+\be b \in D^{*(-1)}$, then
$$
|\chi_{a,b}(S)|^2=q-\frac{q+1}{4}-\frac{q+1}{4}+\frac{(1-2\sqrt{-q}-q)}{4}+\frac{(1+2\sqrt{-q}-q)}{4}=0.
$$
Therefore,
$$
|\chi_{a,b}(S)|^2=\begin{cases}
  q^2 & \mbox{if $a=b=0$,} \\
  \frac{q^2+q}{4} & \mbox{if $a+\al b=0$ or $a+\be b=0$,} \\
  q & \mbox{if $a+\al b \in D^*,a+\be b \in D^{*(-1)}$ or $a+\al b \in D^{*(-1)},a+\be b \in D^{*}$,} \\
  0 & \mbox{if $a+\al b \in D^*,a+\be b \in D^{*}$ or $a+\al b \in D^{*(-1)},a+\be b \in D^{*(-1)}$.}
\end{cases}
$$
Hence, $S$ and $T$ form a formally dual pair. By the above equation and Lemma~\ref{lem-priset}, $S$ is not contained in a coset of a proper subgroup in $G$. Note that $\nu_S((a,b))=q$ if and only if $(a,b)=(0,0)$. This implies $|\chi_{a,b}(T)|^2=q^2=|T|^2$ if and only if $(a,b)=(0,0)$. By Lemma~\ref{lem-priset}, $T$ is not contained in a coset of a proper subgroup in $G$. Thus, $S$ and $T$ form a primitive formally dual pair in $G$.
\end{proof}

\begin{remark}\label{rem-semi}
We remark that the primitive formally dual pairs in Theorem~\ref{thm-SHDS} are inequivalent to those in Theorem~\ref{thm-RDS}(2), since they have distinct character spectra. Moreover, there are numerous constructions of skew Hadamard difference sets in elementary abelian groups, see \cite{DPW,DWX,DY,FX,Paley,WQWX} for instance. Notably, skew Hadamard difference sets can be derived from presemifields with commutative multiplication \cite[Proposition 3.6]{DY}, \cite[Corollary 2.7]{WQWX}. Therefore, many different primitive formally dual pairs can be produced via Theorem~\ref{thm-SHDS}. Hence, a natural question is to determine whether these different primitive formally pairs are equivalent or not.
\end{remark}

We conclude this section by mentioning an unexpected example of primitive formally dual pair in $\Z_2 \times \Z_4 \times \Z_4$.

\begin{example}\label{exam-244}
In the group $\Z_2\times\Z_4 \times \Z_4$, we find a primitive formally dual pair $S$ and $T$, with
$$
S=\{(0,0,0), (0,0,1),(0,1,0), (1,1,1)\}
$$
and
$$
T=\{ (0,0,0), (0,0,1), (0,1,0), (0,1,1), (1,0,0), (1,0,3), (1,3,0), (1,3,3) \}.
$$
It is worth noting that this example is the first known primitive formally dual pair, in which the sizes of two primitive formally dual sets are not equal. We remark that this example has been extended to infinite families in a recent work \cite{LP}.
\end{example}

\section{Even set: a new model}\label{sec4}

The constructions in the last section suggest the close connection between primitive formally dual sets and the RDS-related structures. A natural question is the exact relation between a primitive formally dual set and an $(n,n,n,1)$-RDS. In this section we answer this question by proposing the concept of even sets as a new model for the study of formally dual sets, which includes RDSs and GRDSs as special cases. As we shall see, even sets provide more structural information about primitive formally dual sets. In addition, we shed some new light on Conjecture~\ref{conj-main}, by observing its connection to the nonexistence result of RDSs in cyclic groups.

\subsection{Even set and formally dual set}


We first propose the concept of even sets.

\begin{definition}\label{def-evenset}
Let $G$ be a group and $H_i$ be a subgroup of $G$ for each $1 \le i \le r$. Let $S$ be a subset of $G$, satisfying
\begin{equation}\label{eqn-Sdecom}
SS^{(-1)} = \sum_{i=1}^r \lambda_i H_i,
\end{equation}
where $\la_i$ is a nonzero integer for each $1 \le i \le r$. Then $S$ is called an even set with respect to $H_i$, $1 \le i \le r$. Moreover, if $r$ is the smallest integer such that a decomposition \eqref{eqn-Sdecom} holds, we call $r$ the rank of the even set $S$.
\end{definition}

We remark that Definition~\ref{def-evenset} covers many well-studied configurations as special cases.

\begin{example}
Let $D$ be a $(v,k,\la)$ difference set in a group $G$, satisfying $1 < k < v$. By Definition~\ref{def-DS}, we have
$$
DD^{(-1)}=k+\la (G-1)=(k-\la) + \la G.
$$
Note that $k > \la$ and $\la>0$. Therefore, $D$ is a rank two even set in $G$, with respect to subgroups $\{1\}$ and $G$.
\end{example}

\begin{example}\label{exam-RDS}
Let $R$ be an $(m,n,k,\la)$-RDS in a group $G$ relative to $N$ with $n>1$ and $k>1$. By Definition~\ref{def-RDS}, we have
$$
RR^{(-1)}=k+\la (G-N)=k-\la N+\la G,
$$
where $\la>0$ and $N \ne \{1\}$, $N \ne G$. Thus, $R$ is a rank three even set in $G$, with respect to subgroups $\{1\}$, $N$ and $G$.
\end{example}

\begin{example}
Let $R=\GR(p^t,s)$ and let $S$ be a GRDS in $R \times R$ relative to $\{R_i \times (p^i) \mid 0 \le i \le t\}$. By Definition~\ref{def-GRDS}, we have
$$
SS^{(-1)}=\sum_{i=0}^t p^{is} (R_i \times (p^i)).
$$
For $0 \le i,j \le t$, we know that $(p^i) \times (p^j)$ is a subgroup of $R \times R$. Note that $(p^0)=(1)=R$ and define $(p^{t+1})=\es$. Since
\begin{align*}
&SS^{(-1)}\\
=&\sum_{i=0}^t p^{is} (R_i \times (p^i))=\sum_{i=0}^t (p^{is} ((p^i) \times (p^i))-p^{is} ((p^{i+1}) \times (p^i))) \\
=&p^{ts} ((p^t) \times (p^t))-p^{(t-1)s} ((p^t) \times (p^{t-1}))+p^{(t-1)s} ((p^{t-1}) \times (p^{t-1}))-\cdots-(p) \times R+R \times R,
\end{align*}
$S$ is an even set of rank $2t+1$ in $R \times R$, with respect to a chain of subgroups
$$
\{(0,0)\}=(p^t) \times (p^t) \lneqq (p^t) \times (p^{t-1}) \lneqq (p^{t-1}) \times (p^{t-1}) \lneqq \cdots \lneqq (p) \times R \lneqq R \times R.
$$
\end{example}

For $y \in G$ of order $l$, it generates a cyclic group $\lan y \ran$ of order $l$. We define the \emph{orbit} generated by $y$ as
$$
\orb(y)=\{ y^i \mid i \in \Zl^* \},
$$
where $\Zl^*$ is the set of multiplicative units in $\Zl$. Namely, $\orb(y)$ consists of all elements in $\lan y \ran$ having the same order as $y$. The following proposition presents an alternative description of even sets, for which we give a sketch of proof.

\begin{proposition}\label{prop-orbit}
Let $S$ be a subset of a group $G$. Then $S$ is an even set in $G$ if and only if for each $y \in \{SS^{(-1)}\}$, the weight enumerator $\nu_S$ is constant on $\orb(y)$. Equivalently, $S$ is an even set in $G$ if and only if the multiset $[SS^{(-1)}]$ is a union of orbits in $G$.
\end{proposition}
\begin{proof}[Sketch of proof]
For $x, y \in G$, define a relation $x \sim y$ if $x \in \orb(y)$. It is routine to verify that this relation is an equivalence relation. Thus, the group $G$ and each subgroup of $G$, can be partitioned into a disjoint union of orbits.

If $S$ is an even set in $G$, then $SS^{(-1)}=\sum_{i}\la_iH_i$, where the $H_i$'s are subgroups of $G$. Since each $H_i$ can be partitioned into a disjoint union of orbits, then $[SS^{(-1)}]$ is a union of orbits. Thus, for each $y \in \{SS^{(-1)}\}$, we know that $\nu_S$ is constant on $\orb(y)$.

Conversely, if $\nu_S$ is constant on $\orb(y)$ for each $y \in \{SS^{(-1)}\}$, then $[SS^{(-1)}]$ is a union of orbits. For each orbit $\orb(y)$, by the inclusion and exclusion theorem, we can see that $\orb(y)=\sum_{i}\mu_iL_i$, where $L_i$ is a subgroup of $\lan y \ran$. Consequently, $SS^{(-1)}$ can be decomposed into a summation of subgroups of $G$ and therefore, $S$ is an even set in $G$.
\end{proof}

In order to build the connection between formally dual sets and even sets, we need Lemma~\ref{lem-char}. Each $i \in \Zl^*$ corresponds to a Galois automorphism $\sig_i \in \Gal(\Q(\zl)/\Q)$, where $\sig_i(\zl)=\zl^i$ and $\zl$ is a primitive $l$-th root of unity. Recall that for each $\chi_y \in \wh{G}$, we have a natural projection $\rho: G \rightarrow G/\ker \chi_y$. The character $\chi_y$ induces a character $\wti{\chi_y}$ of $G/\ker \chi_y$, satisfying $\wti{\chi_y}(\rho(g))=\chi_y(g)$ for each $g \in G$.

\begin{lemma}\label{lem-char}
Let $G$ be a group and $y \in G$ of order $l$. Let $\De$ be an isomorphism from $G$ to $\wh{G}$, such that $\De(g)=\chi_g$ for each $g \in G$. Let $\rho: G \rightarrow G/\ker \chi_y$ be the natural projection. Then
\begin{itemize}
\item[(1)] $\chi_y$ is a character of order $l$. For each $i \in \Zl^*$ and $g \in G$, we have $\sig_i(\chi_y(g))=\chi_{y}(g^i)$.
\item[(2)] $\wti{\chi_y}$ has order $l$ and generates the character group $\wh{G/\ker \chi_y}$, which is isomorphic to $\Zl$. For each $i \in \Zl^*$ and $h \in G/\ker \chi_y$, we have $\sig_i(\wti{\chi_{y}}(h))=\wti{\chi_{y}}(h^i)$.
\end{itemize}
\end{lemma}
\begin{proof}
(1) Since $\De$ is a group isomorphism, $\chi_y$ has order $l$. For each $i \in \Zl^*$ and $g \in G$, we have
$$
\sig_i(\chi_y(g))=\chi_y^i(g)=\chi_{y}(g^i).
$$

(2) By the definition of $\wti{\chi_y}$, we can easily see that $\wti{\chi_y}$ has order $l$. Note that $\chi_y: G \rightarrow \{\zl^j \mid 0 \le j \le l-1\}$ is an epimorphism. Therefore, $G/ \ker \chi_y \cong \{\zl^j \mid 0\le j \le l-1\} \cong \Z_l$. Consequently, the character group $\wh{G/\ker \chi_y}$ is isomorphic to $\Zl$ and has a generator $\wti{\chi_y}$. For each $i \in \Zl^*$ and $h \in G$, we have
    $\sig_i(\wti{\chi_y}(h))=\wti{\chi_{y}}^i(h)=\wti{\chi_{y}}(h^i)$.
\end{proof}

Let $S$ and $T$ be a formally dual pair in $G$. By Remark~\ref{rem-def}(3), $|\chi(S)|^2$ and $|\chi(T)|^2$ are integers for each $\chi \in \wh{G}$. The following proposition extends Proposition~\ref{prop-non}(3).

\begin{proposition}\label{prop-decom}
Let $S$ be a subset of a group $G$, such that $|\chi(S)|^2$ is an integer for each $\chi \in \wh{G}$. Let $y \in G$ be an element of order $l$. Then the following holds:
\begin{itemize}
\item[(1)] For each $i \in \Z_l^*$, we have $|\chi_y(S^{(i)})|^2=|\chi_y(S)|^2$.
\item[(2)] Furthermore, suppose $S$ and $T$ form a formally dual pair in $G$. Then for each $i \in \Z_l^*$, we have $\nu_S(y^i)=\nu_S(y)$ and $\nu_T(y^i)=\nu_T(y)$.
\end{itemize}
\end{proposition}
\begin{proof}
(1) By Lemma~\ref{lem-char}(1), $\chi_y$ has order $l$. For each $i \in \Zl^*$, we have
\begin{align*}
|\chi_y(S)|^2&=\sig_i(|\chi_y(S)|^2)=\sig_i(\chi_y(S))\sig_i(\ol{\chi_y(S)})\\
             &=\sig_i(\chi_y(S))\ol{\sig_i(\chi_y(S))}=\chi_{y}^i(S)\ol{\chi_{y}^i(S)}=\chi_{y}(S^{(i)})\ol{\chi_{y}(S^{(i)})}=|\chi_{y}(S^{(i)})|^2.
\end{align*}

(2) Let $\De$ be the isomorphisms from $G$ to $\wh{G}$, such that $\De(g)=\chi_g$ for each $g \in G$. For each $y \in G$ of order $l$ and each $i \in \Zl^*$, note that $\chi_{y^i}=\De(y^i)=\De(y)^i=\chi_y^i$. Applying (1), we have
$$
|\chi_{y^i}(S)|^2=|\chi_y^i(S)|^2=|\chi_y(S^{(i)})|^2=|\chi_y(S)|^2.
$$
By formal duality and \eqref{eqn-def}, for each $y \in G$ of order $l$ and $i \in \Zl^*$, we have $\nu_T(y^i)=\nu_T(y)$. Interchanging the roles of $S$ and $T$, for each $y \in G$ of order $l$ and $i \in \Zl^*$, we have $\nu_S(y^i)=\nu_S(y)$.
\end{proof}

As an immediate consequence of Proposition~\ref{prop-orbit} and Proposition~\ref{prop-decom}(2), we have the following corollary.

\begin{corollary}\label{cor-formaleven}
Each formally dual set is an even set.
\end{corollary}

We define the \emph{rank} of a formally dual set to be its rank as an even set. In the sense of Definition~\ref{def-equiv}, two equivalent formally dual sets have the same rank. If $S$ is a formally dual set in $G$, which is an even set with respect to subgroups $H_i$, $1 \le i \le r$, then we call $S$ a formally dual set in $G$ with respect to subgroups $H_i$, $1 \le i \le r$.

The following proposition describes a necessary and sufficient condition for a pair of even sets being formally dual.

\begin{theorem}\label{thm-convertion}
Let $G$ be a group. The following two statements are equivalent.
\begin{itemize}
\item[(1)] $S$ and $T$ form a formally dual pair in $G$.
\item[(2)] $S$ and $T$ are even sets of rank $r$, such that
\begin{equation}\label{eqn-SSinv}
SS^{(-1)}=\sum_{i=1}^r \lambda_i H_i
\end{equation}
and
\begin{equation}\label{eqn-TTinv}
TT^{(-1)}=\sum_{i=1}^{r} \wti{\lambda_i} \wti{H_i},
\end{equation}
where $H_i$, $1 \le i \le r$, are subgroups of $G$ and $\wti{\la_i}=\frac{\lambda_i|G||H_i|}{|S|^3}$, for each $1 \le i \le r$.
\end{itemize}
\end{theorem}
\begin{proof}
We first prove that (1) implies (2). Since $S$ and $T$ form a formally dual pair, then by Corollary~\ref{cor-formaleven}, $S$ and $T$ are both even sets. Suppose $S$ has rank $r$, then we can assume
$$
SS^{(-1)}=\sum_{i=1}^r \lambda_i H_i,
$$
where $H_i$, $1 \le i \le r$, are subgroups of $G$ and $\la_i$, $1 \le i \le r$, are nonzero. Note that for each $1 \le i \le r$,
\begin{equation}\label{eqn-ortho}
\chi_y(H_i)=\begin{cases}
                 |H_i| & \mbox{if $y \in \wti{H_i}$,}\\
                  0    & \mbox{otherwise.}
            \end{cases}
\end{equation}
For each $y \in G$, we define a subset $I_y \subset \{1,2,\ldots,r\}$, such that $y \in \big(\bigcap_{i \in I_y} \wti{H_i}\big) \sm \big(\bigcup_{j \notin I_y} \wti{H_j}\big)$. Combining \eqref{eqn-def}, \eqref{eqn-SSinv} and \eqref{eqn-ortho}, we have
$$
\nu_T(y)=\frac{|T|}{|S|^2}|\chi_y(S)|^2 =\frac{|G|}{|S|^3}\sum_{i=1}^r \lambda_i \chi_y(H_i)=\frac{|G|}{|S|^3}\sum_{i \in I_y} \lambda_i|H_i|=\sum_{i \in I_y} \wti{\la_i},
$$
which implies
$$
TT^{(-1)}=\sum_{i=1}^{r} \wti{\lambda_i} \wti{H_i}.
$$
Thus, the rank of $T$ is no more than $r$. By interchanging the roles of $S$ and $T$, we can see the ranks of $S$ and $T$ are both $r$.

Conversely, assume (2) holds. An analogous argument shows that (2) implies (1).
\end{proof}

\begin{remark}
Let $S$ and $T$ be a formally dual pair in $G$. By Remark~\ref{rem-def}(2), $SS^{(-1)}$ determines $TT^{(-1)}$, and vice versa. Based on the concept of even sets, \eqref{eqn-SSinv} and \eqref{eqn-TTinv} provide an explicit description about how $SS^{(-1)}$ and $TT^{(-1)}$ interact.
\end{remark}

The following lemma describes more structural restrictions on a primitive formally dual set.

\begin{lemma}\label{lem-prieven}
Let $S$ and $T$ be a primitive formally dual pair in $G$, such that $SS^{(-1)}=\sum_{i=1}^r \lambda_i H_i$, where $H_i$ is a subgroup of $G$ for each $1 \le i \le r$. Then we have
\begin{itemize}
\item[(1)] The group generated by $H_1, H_2, \ldots, H_r$ is $G$.
\item[(2)] $\bigcap_{i=1}^r H_i= \{1\}$.
\item[(3)] $\sum_{i=1}^r\la_i=|S|$ and $\sum_{i=1}^r\la_i|H_i|=|S|^2$
\end{itemize}
\end{lemma}
\begin{proof}
(1) Let $H$ be the group generated by $H_1, H_2, \ldots, H_r$. Assume $H$ is a proper subgroup of $G$. Since
$$
\{SS^{(-1)}\} \subset H,
$$
for each nonprincipal $\chi \in \wh{G}$, which is principal on $H$, we have $|\chi(S)|^2=|S|^2$. By Lemma~\ref{lem-priset}, $S$ is not a primitive subset, which is a contradiction.

(2) By Theorem~\ref{thm-convertion}, we have
$$
TT^{(-1)}=\sum_{i=1}^{r} \wti{\lambda_i} \wti{H_i},
$$
where $\wti{\la_i}=\frac{\lambda_i|G||H_i|}{|S|^3}$, for each $1 \le i \le r$. Since the group generated by $H_1, H_2, \ldots, H_r$ is the whole group $G$, by definition of $\wti{H_i}$, we have $\bigcap_{i=1}^r \wti{H_i}=\{1\}$. By interchanging the roles of $S$ and $T$, we conclude that $\bigcap_{i=1}^r H_i= \{1\}$.

(3) Note that $|S|=[SS^{(-1)}]_1=\sum_{i=1}^r\la_i$ and $|S|^2=|[SS^{(-1)}]|=\sum_{i=1}^r\la_i|H_i|$, the conclusion follows.
\end{proof}

For a primitive formally dual set with respect to a chain of subgroups, we can derive much more detailed information.

\begin{proposition}\label{prop-chain}
Let $S$ and $T$ be a primitive formally dual pair in $G$. Let $S$ be a primitive formally dual set in $G$, with respect to a chain of subgroups $H_1 \lneqq H_2 \lneqq \cdots \lneqq H_r$. Then $T$ is a primitive formally dual set in $G$, with respect to a chain of subgroups $\wti{H_r} \lneqq \wti{H_{r}}_{-1} \lneqq \cdots \lneqq \wti{H_1}$. We write $SS^{(-1)}=\sum_{i=1}^r \lambda_i H_i$ and $TT^{(-1)}=\sum_{i=1}^r \wti{\la_i} \wti{H_i}$. Furthermore, the following holds.
\begin{itemize}
\item[(1)] $H_1=\wti{H_r}= \{1\}$ and $H_r=\wti{H_1}=G$.
\item[(2)] For each $2 \le l \le r$, we have $0 \le \sum_{i=l}^r\la_i < |S|$. For each $1 \le l \le r-1$, we have $ 0 \le \sum_{i=1}^l\la_i|H_i| < |S|^2$. Moreover,
$$
\la_1=|S|, \quad \la_r=1, \quad |S| \Big| \la_i|H_i|, \; \forall 1 \le i \le r.
$$
\item[(3)] For each $1 \le l \le r-1$, we have $0 \le \sum_{i=1}^l\wti{\la_i} < |T|$. For each $2 \le l \le r$, we have $0 \le \sum_{i=l}^r\wti{\la_i}|\wti{H_i}| < |T|^2$. Moreover,
$$
\wti{\la_1}=1, \quad \wti{\la_r}=|T|, \quad |T| \Big| \wti{\la_i}|\wti{H_i}|, \; \forall 1 \le i \le r.
$$
\item[(4)] $|G|$ is a square and $|S|=|T|=\sqrt{|G|}$.
\end{itemize}
\end{proposition}
\begin{proof}
(1) Note that $H_1 \lneqq H_2 \lneqq \cdots \lneqq H_r$ and $\wti{H_r} \lneqq \wti{H_{r}}_{-1} \lneqq \cdots \lneqq \wti{H_1}$. Then the conclusion follows from Lemma~\ref{lem-prieven}(1)(2).

Now we proceed to prove (2) and (4). For $2 \le l \le r$, choose $g \in H_l$ such that $g \notin H_{l-1}$. Together with the primitivity of $S$, we have $0 \le [SS^{(-1)}]_g=\sum_{i=l}^r \la_i < |S|$. For $1 \le l \le r-1$, let $\chi \in \wh{G}$ be a character which is principal on $H_l$ and nonprincipal on $H_{l+1}$. Together with the primitivity of $S$, we have $0 \le |\chi(S)|^2=\sum_{i=1}^l\la_i|H_i| < |S|^2$. Thus, we have proved the first two statements of (2), which imply that $1 \le \la_1, \la_r \le |S|$. By interchanging the roles of $S$ and $T$, we have $1 \le \wti{\la_1}, \wti{\la_r} \le |T|$. Recall that $\wti{\la_i}=\frac{\la_i|G||H_i|}{|S|^3}$ for each $1 \le i \le r$. We have
$$
1 \leq \wti{\la_1} = \frac{\la_1|G||H_1|}{|S|^3}=\frac{\la_1|G|}{|S|^3} \leq \frac{|G|}{|S|^2}
$$
and
$$
1 \leq \la_r = \frac{\wti{\la_r}|S|^3}{|G||H_r|}=\frac{\wti{\la_r}|S|^3}{|G|^2} \leq \frac{|S|^2}{|G|},
$$
which implies $\la_1=|S|$, $\la_r=1$ and $|G|=|S|^2$. Together with $|G|=|S||T|$, we have $|G|$ is a square and $|S|=|T|=\sqrt{|G|}$. Since $\wti{\la_i}=\frac{\la_i|G||H_i|}{|S|^3}=\frac{\la_i|H_i|}{|S|}$, we have $|S| \Big| \la_i|H_i|$, for each $1 \le i \le r$.

(3) can be proved in a similar manner as above.
\end{proof}

\begin{remark}\label{rem-equalsize}
We observe that Proposition~\ref{prop-chain}(4) still holds under a weaker condition:
\begin{equation}\label{eqn-weaker}
H_1=\{1\}, \quad H_r=G, \quad |\bigcap_{i=2}^r H_i|>1.
\end{equation}
Note that we have a chain of subgroups $H_1 \lneqq H_2 \lneqq \cdots \lneqq H_r$ in Proposition~\ref{prop-chain}, which implies \eqref{eqn-weaker}. Moreover, the primitive formally dual pairs derived in Theorems~\ref{thm-RDS} and~\ref{thm-GRDS} fit into the model of Proposition~\ref{prop-chain}. On the other hand, the primitive formally dual pairs in Theorem~\ref{thm-SHDS} seem to be more complicated.
\end{remark}

\subsection{Primitive formally dual sets with small rank}\label{sec:fd_small_rank}

In this subsection we focus on primitive formally dual sets with small rank. We give a characterization of primitive formally dual sets with rank at most three. In particular, the rank three case reveals the relation between a primitive formally dual set and an $(n,n,n,1)$-RDS.

Let $S$ be a primitive formally dual set in $G$, with respect to $H_i$, $1 \le i \le r$. The following definition provides a crucial viewpoint on the relation among the subgroups $H_i$.

\begin{definition}
Let $H_i$, $1 \le i \le r$, be a collection of subgroups in $G$. A subgroup $H_j$ is called maximal, if it is not a subgroup of any other $H_i$. A subgroup $H_j$ is called minimal, if it does not contain any other $H_i$ as a subgroup.
\end{definition}

The concept of maximal or minimal subgroups reveals more structural information about primitive formally dual sets.

\begin{lemma}\label{lem-twomaximal}
Let $S$ be a primitive formally dual set in $G$, with respect to $H_i$, $1 \le i \le r$. Then $H_i$, $1 \le i \le r$, do not have exactly two maximal subgroups or exactly two minimal subgroups.
\end{lemma}
\begin{proof}
Suppose otherwise that $H_1,H_2$ are exactly the two maximal subgroups among $H_i$, $1 \le i \le r$. Thus, for each $s_1,s_2 \in S$, we must have $s_1s_2^{-1} \in H_1 \cup H_2$. Given a $z \in S$, we claim that there exist $x\in S$ such that $ xz^{-1} \in H_1 \sm H_2$. Suppose otherwise, that $xz^{-1} \in H_2$ for each $x \in S$, then $S$ is contained in the coset $zH_2$ and $S$ is not primitive. Similarly, we can show that there must be $y\in S$ such that $zy^{-1} \in H_2 \sm H_1$. However, we have
$$
xy^{-1} = (xz^{-1})(zy^{-1}) \notin H_1\cup H_2,
$$
which yields a contradiction.

Suppose otherwise that $H_1,H_2$ are exactly the two minimal subgroups among $H_i$, $1 \le i \le r$. Let $S$ and $T$ be a primitive formally dual pair in $G$. By Theorem~\ref{thm-convertion}, $T$ is a primitive formally dual set in $G$, with respect to $\wti{H_i}$, $1 \le i \le r$. which contains exactly two maximal subgroups $\wti{H_1}, \wti{H_2}$. This leads to a contradiction against the above conclusion.
\end{proof}

We shall see that it is easy to characterize primitive formally dual sets of rank one and two.

\begin{theorem}
The only primitive formally dual set of rank one is $\{1\}$ in the trivial group $G = \{1\}$.
\end{theorem}
\begin{proof}
Let $S$ be a primitive formally dual set of rank one in $G$, with respect to $H$, satisfying $SS^{(-1)} = \la H$. By Lemma~\ref{lem-prieven}(1)(2), we have $H=G=\{1\}$. Thus, $S=\{1\}$.
\end{proof}

\begin{theorem}
There exists no primitive formally dual set of rank two.
\end{theorem}
\begin{proof}
Let $S$ be a primitive formally dual set of rank two in $G$, with respect to $H_1,H_2$. By Lemma~\ref{lem-twomaximal}, $H_1$ and $H_2$ must form a chain, say, $H_1 \lneqq H_2$. By Proposition~\ref{prop-chain}(1)(2), we have $H_1=\{1\}$, $H_2=G$ and
$$
SS^{(-1)}=|S|+G.
$$
This is impossible since the coefficients of $1$ on both sides do not match.
\end{proof}

Now we proceed to consider the characterization of rank three primitive formally dual sets. Let $S$ be a formally dual set of rank $r$ in $G$, with respect to $H_i$, $1 \le i \le r$. Let $I$ be a subset of $\{1,2,\ldots,r\}$, define
\begin{equation}\label{eqn-defSI}
S_I=S \bigcap \Big(\bigcap_{i \in I}H_i\Big) \sm \Big(\bigcup_{j \notin I} H_j\Big).
\end{equation}
To simplify our notation, for instance, we may simply use $S_{12}$ to represent the set $S_{\{1,2\}}$. The following is a preparatory lemma.

\begin{lemma}\label{lem-rankthree}
Let $S$ be an even set of rank three, with respect to three subgroups $H_1, H_2, H_3$, so that each of $H_t$, $1 \le t \le 3$, is maximal. Let $SS^{(-1)}=\sum_{t=1}^3 \la_tH_t$ and $\{i,j,l\}=\{1,2,3\}$. Then we have the following.
\begin{itemize}
\item[(1)] One of $S_i$ and $S_{jl}$ is empty.
\item[(2)] Suppose $S_i \ne \es$. Then for $x,y \in S_j$, we have $xy^{-1} \in H_j \cap H_l$. Moreover, if $S_1$, $S_2$ and $S_3$ are all nonempty, then $|S_1|=|S_2|=|S_3|=1$.
\item[(3)] Suppose $S_i,S_j \ne \es$ and $S_l=\es$. Then
\begin{equation}\label{eqn-SIprop}
\begin{aligned}
&\{S_i+S_i^{(-1)}\} \subset H_i \sm (H_j \cup H_l), \quad \{S_{ij}+S_{ij}^{(-1)}\} \subset (H_i \cap H_j) \sm \{1\}, \\
&\{S_iS_i^{(-1)}\} \subset H_i \cap H_l, \quad \{S_{ij}S_{ij}^{(-1)}\} \subset H_i \cap H_j, \\
&\{S_iS_j^{(-1)}+S_jS_i^{(-1)}\} \subset H_l \sm (H_i \cup H_j), \\
&\{S_iS_{ij}^{(-1)}+S_{ij}S_i^{(-1)}\} \subset H_i \sm H_j=((H_i \cap H_l) \sm \{1\}) \cup (H_i \sm (H_j \cup H_l)). \\
\end{aligned}
\end{equation}
Moreover, for each $z \in H_l \sm (H_i \cup H_j)$, we have $[S_iS_j^{(-1)}]_z \le 1$.
\item[(4)] $\la_t>0$ for $1 \le t \le 3$ and $\{SS^{(-1)}\} \cap (H_i \sm (H_j \cup H_l)) \ne \es$.
\end{itemize}
\end{lemma}
\begin{proof}
Since $S$ has rank three and $H_1,H_2,H_3$ are three maximal subgroups, then for any $x,y \in S$, we have $xy^{-1} \in H_1 \cup H_2 \cup H_3$. Recall that by Lemma~\ref{lem-prieven}(2), $H_1 \cap H_2 \cap H_3=\{1\}$. Without loss of generality, we assume $1 \in S$.

(1) Suppose otherwise, that $x \in S_i$ and $y \in S_{jl}$. Then $xy^{-1} \notin H_1 \cup H_2 \cup H_3$, which contradicts the property of even sets.

(2) Suppose $z \in S_i$. Then $xz^{-1}, zy^{-1} \in H_l \sm (H_i \cup H_j)$. Thus, $xy^{-1}=(xz^{-1})(zy^{-1}) \in H_j \cap H_l$. Moreover, if $S_1$, $S_2$ and $S_3$ are all nonempty, there exists $w \in S_l$. A similar argument shows that $xy^{-1}=(xw^{-1})(wy^{-1}) \in H_i \cap H_j$. Thus, $xy^{-1} \in H_1\cap H_2 \cap H_3=\{1\}$ and $x=y$. Therefore, $|S_j|=1$. By the choice of $j$, we have $|S_1|=|S_2|=|S_3|=1$.

(3) \eqref{eqn-SIprop} follows from \eqref{eqn-defSI} and (2). Assume there exists $z \in H_l \sm (H_i \cup H_j)$, such that $[S_iS_j^{(-1)}]_z>1$. Then, there exist distinct $x,y \in S_j$, such that $xz,yz \in S_i$ and $(xz)x^{-1},(yz)y^{-1} \in S_iS_j^{(-1)}$. Then by \eqref{eqn-SIprop}, we have $xy^{-1} \in \{S_jS_j^{(-1)}\} \subset H_j \cap H_l$ and $xy^{-1}=(xz)(yz)^{-1} \in \{S_iS_i^{(-1)}\} \subset H_i \cap H_l$. Therefore $xy^{-1}\in H_i \cap H_j \cap H_l=\{1\}$ and $x=y$, which gives a contradiction. Thus, for each $z \in H_l \sm (H_i \cup H_j)$, we have $[S_iS_j^{(-1)}]_z \le 1$.


(4) Recall that $SS^{(-1)}=\sum_{t=1}^3 \la_tH_t$. Since $H_1$, $H_2$ and $H_3$ are three maximal subgroups, there exists a nonprincipal character $\chi$ which is principal on $H_1$ and nonprincipal on $H_2$ and $H_3$. Thus, $|\chi(S)|^2=\la_1|H_1| \ge 0$, which implies $\la_1>0$. Similarly, we can show that $\la_t>0$ for each $1 \le t \le 3$. By the maximality of subgroups, there exists $x \in H_i \sm H_j$ and $y \in H_i \sm H_l$. If $x=y$, then $x \in H_i \sm (H_j \cup H_l)$ and $x=x1^{-1} \in \{SS^{(-1)}\}$, which implies $\{SS^{(-1)}\} \cap (H_i \sm (H_j \cup H_l)) \ne \es$. If $x \ne y$, then $xy^{-1} \in H_i \sm (H_j \cup H_l)$, which also implies $\{SS^{(-1)}\} \cap (H_i \sm (H_j \cup H_l)) \ne \es$.
\end{proof}

For primitive formally dual sets of rank three, we have the following characterization.

\begin{theorem}\label{thm-rankthree}
If $S$ is a primitive formally dual set of rank three in $G$, then $S$ is an $(n,n,n,1)$-RDS in $G$ with $n>1$.
\end{theorem}
\begin{proof}

Suppose $S$ is a primitive formally dual set of rank three in group $G$, with respect to three subgroups $H_i$, $1 \le i \le 3$. Then there exists a primitive formally dual pair formed by $S$ and $T$ in $G$, and both $S$ and $T$ have rank three. Throughout the proof, we assume without loss of generality that $|G| \le |S|^2$. Otherwise, if $|G|>|S|^2$, we can switch the roles of $S$ and $T$. Let
\begin{equation}\label{eqn-rankthree}
SS^{(-1)}=\la_1H_1+\la_2H_2+\la_3H_3,
\end{equation}
where $\la_i$, $1 \le i \le 3$, are nonzero integers. By Lemma~\ref{lem-twomaximal}, there are either one or three maximal subgroups among $H_1,H_2,H_3$ and either one or three minimal subgroups among $H_1,H_2,H_3$. Therefore, either $H_1,H_2,H_3$ form a chain $H_1 \lneqq H_2 \lneqq H_3$, or each of them is both maximal and minimal.

Consider the case $H_1\lneqq H_2\lneqq H_3$. We use $n$ to denote the size of $S$. Then by Lemma~\ref{lem-prieven} and Proposition~\ref{prop-chain}, we can easily derive that
$$
|G|=n^2, \quad H_1=\{1\}, \quad |H_2|=n, \quad H_3=G, \quad \la_1=n, \quad \la_2=-1, \quad \la_3=1.
$$
Consequently, we have $n>1$ and
$$
SS^{(-1)}=n-H_2+G,
$$
where $H_2$ is a subgroup of $G$ of order $n$, and therefore, $S$ is an $(n,n,n,1)$-RDS in $G$ relative to $H_2$, with $n>1$.

Below, we assume that each subgroup $H_i$, $1 \le i \le 3$ is both maximal and minimal. Without loss of generality, we assume that $1 \in S$, thus $S \subset H_1 \cup H_2 \cup H_3$ and $S=1+S_1+S_2+S_3+S_{12}+S_{13}+S_{23}$. We claim that $S_1 \ne \es$. In fact, by Lemma~\ref{lem-rankthree}(4), there exists $x,y \in S$, such that $xy^{-1} \in H_1 \sm (H_2 \cup H_3)$. If $y=1$, then $x \in S_1 \ne \es$. Otherwise, consider the translation $y^{-1}S$, we have $1 \in y^{-1}S$ and $xy^{-1} \in y^{-1}S \cap (H_1 \sm (H_2 \cup H_3))$. Replacing $S$ with $y^{-1}S$, we have $1 \in y^{-1}S$ and $(y^{-1}S)_1 \ne \es$. Thus, by taking a proper translation if necessary, we can assume $S_1 \ne \es$.

Next, we split our discussion into three cases.

Case I: $S_2 \ne \es$ and $S_3 \ne \es$. By Lemma~\ref{lem-rankthree}(1), we have $S_{12}=S_{13}=S_{23}=\es$. So, $S=1+S_1+S_2+S_3$. By Lemma~\ref{lem-rankthree}(2), we have $|S_{1}|=|S_{2}|=|S_{3}|=1$ and therefore, $|S|=4$. By Lemma~\ref{lem-rankthree}(4), $\la_1,\la_2,\la_3>0$. Note that $\la_1+\la_2+\la_3=4$, we may assume without loss of generality that $\la_1=\la_2=1$ and $\la_3=2$. Let $S=1+a_1+a_2+a_3$, where $a_i \in S_i$, $1 \le i \le 3$. Computing $SS^{(-1)}$, we have
\begin{align*}
1+a_1+a_1^{-1}+a_2a_3^{-1}+a_3a_2^{-1}&=H_1 \\
1+a_2+a_2^{-1}+a_1a_3^{-1}+a_3a_1^{-1}&=H_2 \\
2+a_3+a_3^{-1}+a_1a_2^{-1}+a_2a_1^{-1}&=2H_3
\end{align*}
Thus, $H_3 \cong \Z_3$ and $H_3=1+a_3+a_3^{-1}=1+a_1a_2^{-1}+a_2a_1^{-1}$. If $a_3=a_1a_2^{-1}$, we have $1+2a_1+2a_1^{-1}=H_1$, which is impossible. If $a_3=a_2a_1^{-1}$, we have $1+2a_2+2a_2^{-1}=H_2$, which is impossible.

Case II: $S_2=\es$ and $S_3=\es$. By Lemma~\ref{lem-rankthree}(1), $S=1+S_1+S_{12}+S_{13} \subset H_1$. This leads to a contradiction with the primitivity of $S$.

Case III: Exactly one of $S_2$ and $S_3$ is empty. Without loss of generality, assume $S_3=\es$. By Lemma~\ref{lem-rankthree}(1), we have $S_{13}=S_{23}=\es$. So, $S=1+S_1+S_2+S_{12}$. Therefore,
\begin{align*}
SS^{(-1)}=&1+S_1+S_1^{(-1)}+S_2+S_2^{(-1)}+S_{12}+S_{12}^{(-1)}+S_1S_1^{(-1)}+S_2S_2^{(-1)}+S_{12}S_{12}^{(-1)}\\
          &+S_1S_2^{(-1)}+S_2S_1^{(-1)}+S_1S_{12}^{(-1)}+S_{12}S_1^{(-1)}+S_2S_{12}^{(-1)}+S_{12}S_2^{(-1)}.
\end{align*}
By \eqref{eqn-defSI} and Lemma~\ref{lem-rankthree}, we have
\begin{equation}\label{eqn-SIprop2}
\begin{aligned}
& \{S_1+S_1^{(-1)}\} \subset H_1 \sm (H_2 \cup H_3), \quad \{S_2+S_2^{(-1)}\} \subset H_2 \sm (H_1 \cup H_3),\\
&\{S_{12}+S_{12}^{(-1)}\} \subset (H_1 \cap H_2) \sm \{1\}, \quad \{S_1S_1^{(-1)}\} \subset H_1 \cap H_3, \quad \{S_2S_2^{(-1)}\} \subset H_2 \cap H_3, \\
&\{S_{12}S_{12}^{(-1)}\} \subset H_1 \cap H_2, \quad \{S_1S_2^{(-1)}+S_2S_1^{(-1)}\} \subset H_3 \sm (H_1 \cup H_2), \\
&\{S_1S_{12}^{(-1)}+S_{12}S_1^{(-1)}\} \subset H_1 \sm H_2=((H_1 \cap H_3) \sm \{1\}) \cup (H_1 \sm (H_2 \cup H_3)), \\
&\{S_2S_{12}^{(-1)}+S_{12}S_2^{(-1)}\} \subset H_2 \sm H_1=((H_2 \cap H_3) \sm \{1\}) \cup (H_2 \sm (H_1 \cup H_3)).
\end{aligned}
\end{equation}
By Lemma~\ref{lem-rankthree}(4), we have $\{SS^{(-1)}\} \cap (H_3 \sm (H_1 \cup H_2)) \ne \es$. Suppose $z \in (H_3 \sm (H_1 \cup H_2))$. Then by \eqref{eqn-SIprop2}, we have $[SS^{(-1)}]_z=[S_1S_2^{(-1)}+S_2S_1^{(-1)}]_z$. By Lemma~\ref{lem-rankthree}(3), $[SS^{(-1)}]_z=[S_1S_2^{(-1)}+S_2S_1^{(-1)}]_z=[S_1S_2^{(-1)}]_z+[S_2S_1^{(-1)}]_z \le 2$. Together with \eqref{eqn-rankthree}, we have $[SS^{(-1)}]_z=\la_3 \le 2$.

Suppose $x \in S_1$ and consider $S^{\pr}=x^{-1}S$. It is easy to see that $1 \in S^{\pr}$, $x^{-1} \in S_1^{\pr} \ne \es$, $S_3^{\pr} \ne \es$ and $S_2^{\pr}=\es$. Indeed, for $y \in S_2$, we have $x^{-1}y \in S^{\pr}$ and $x^{-1}y \in H_3 \sm (H_1 \cup H_2)$, which implies $x^{-1}y \in S_3^{\pr} \ne \es$. If $S_2^{\pr} \ne \es$, then there exists $z \in S_2^{\pr}$, which implies $xz \in S_3$, contradicting $S_3=\es$. Since $S^{\pr}{S^{\pr}}^{(-1)}=SS^{(-1)}=\la_1H_1+\la_2H_2+\la_3H_3$, considering $S^{\pr}$ and using an analogous approach as above, we can show that $\la_2 \le 2$. Similarly, suppose $y \in S_2$ and consider $S^{\pr\pr}=y^{-1}S$. Then we can show that $\la_1 \le 2$. Thus, we have $1 \le \la_i \le 2$, $1 \le i \le 3$ and therefore, $3 \le |S| \le 6$.

Thus, the problem has been reduced to finitely many cases, for which we need to characterize rank three primitive formally dual sets $S$ in $G$, with $3 \le |S| \le 6$ and $|G| \le |S|^2 \le 36$.

Let $S$ and $T$ be a primitive formally dual pair in $G$. By Propositions~\ref{prop-non}(1) and Theorem~\ref{thm-convertion}, $T$ is a rank three primitive formally dual set in $G$, satisfying $2 \le |T| \le 6$ and $|T|^2 \le |G| \le 36$. Now we proceed to do a computer search, which leads to Table~\ref{tab-table} in Appendix~\ref{app-A}. Up to equivalence, Table~\ref{tab-table} covers every rank three primitive formally dual set $T$ in $G$ satisfying the above condition, which is marked by $\bigstar$. We can see that each of them is an $(n,n,n,1)$-RDS in $G$ with $n>1$. Moreover, by Definition~\ref{def-equiv}, every set equivalent to $T$ must be an $(n,n,n,1)$-RDS in $G$ with $n>1$. Thus, by Theorem~\ref{thm-RDS}, $S$ is an $(n,n,n,1)$-RDS in $G$ with $n>1$ and we complete the proof.

We note that a theoretical proof tackling the finitely many cases in detail, is presented in Appendix~\ref{app-B}.
\end{proof}

\begin{remark}
By Example~\ref{exam-RDS}, each RDS $S$ in $G$ with $|S|>1$ is an even set of rank three. Being a primitive formally dual set forces $S$ to be an $(n,n,n,1)$-RDS in $G$ with $n>1$.
\end{remark}

Now, we illustrate how the characterization of rank three primitive formally dual sets can provide a new viewpoint towards Conjecture~\ref{conj-main}. We begin with the following proposition, which is a direct consequence of \cite[Theorem 4.1.1]{Pott95}.

\begin{proposition}\label{prop-cycRDS}
Let $R$ be an $(n,n,n,1)$-RDS in cyclic group $\Z_{n^2}$. Then one of the following holds.
\begin{itemize}
\item[(1)] $n=1$ and $R=\{1\}$ is the trivial RDS in the trivial group $G=\{1\}$.
\item[(2)] $n=2$ and up to equivalence, $R=\{1,g\}$ is a $(2,2,2,1)$-RDS in $\Z_4=\{1,g,g^2,g^3\}$.
\end{itemize}
\end{proposition}

We recall that when $N$ is a prime power, Conjecture~\ref{conj-main} has been proved in \cite{Sch} and \cite{Xia}. The author of \cite{Xia} essentially showed that a primitive formally dual set in the cyclic group $\Z_{2^{2k}}$ with $k \ge 1$ must have rank three. By Theorem~\ref{thm-rankthree}, such a primitive formally dual set must be a $(2^k,2^k,2^k,1)$-RDS in $\Z_{2^{2k}}$, for which there is only one example described in Proposition~\ref{prop-cycRDS}(2). In light of this observation, it may not be a coincidence that Proposition~\ref{prop-cycRDS} exactly captures the only two primitive formally dual sets in cyclic groups: Proposition~\ref{prop-cycRDS}(1) corresponds to the trivial configuration and Proposition~\ref{prop-cycRDS}(2) corresponds to the TITO configuration. Finally, we propose the following question.

\begin{question}\label{ques-rankthree}
Is it true that a primitive formally dual set in a cyclic group $\Z_N$ with $N > 1$ must have rank three?
\end{question}

An affirmative answer to Question~\ref{ques-rankthree}, together with Theorem~\ref{thm-rankthree} and Proposition~\ref{prop-cycRDS}, will show that Conjecture~\ref{conj-main} is true.

\section{Nonexistence of primitive formally dual pairs}\label{sec5}

In this section, we derive some nonexistence results about primitive formally dual pairs, which provide more supportive evidence towards Conjecture~\ref{conj-main}.

Let $G$ be a group of order $N$ and $p$ a prime divisor of $N$. We use $G_p$ to denote the Sylow $p$-subgroup of $G$. For a primitive formally dual pair $S$ and $T$, define
$$
a_S=\frac{|S|^2}{(|S|^2,|T|)}, \quad a_T=\frac{|T|^2}{(|T|^2,|S|)}, \quad b_S=\frac{|S|}{(|S|,|T|^2)}, \quad b_T=\frac{|T|}{(|T|,|S|^2)}.
$$

The following theorem indicates that the divisibility of weight enumerators implies some restrictions on the size of $S$ and $T$.

\begin{theorem}\label{thm-order}
Let $S$ and $T$ be a primitive formally dual pair in a group $G$. Then for each $y \in G$, we have $b_S \big| \nu_S(y)$ and $b_T \big| \nu_T(y)$. Let $p$ be a prime divisor of $|G|$ such that $G_p$ is cyclic. Suppose there exists a positive integer $r$, such that $p^r \big\| a_S$. Then we have
$$
\left\lf \frac{|S|}{p^r} \right\rf \ge \frac{b_S}{(b_S,p^r)}.
$$
Suppose there exists a positive integer $r^\pr$, such that $p^{r^{\pr}} \big\| a_T$. Then we have
$$
\left\lf \frac{|T|}{p^{r^\pr}} \right\rf \ge \frac{b_T}{(b_T,p^{r^{\pr}})}.
$$
\end{theorem}
\begin{proof}
By Remark~\ref{rem-def}(3), in order to ensure that $|\chi_y(T)|^2$ being an integer, for each $y \in G$, we must have $b_S \big| \nu_S(y)$. Similarly, for each $y \in G$, we have $b_T \big| \nu_T(y)$. Below, we only prove the first inequality since the second one is analogous. By \eqref{eqn-def}, $p^r \big\| a_S$ implies $p^r \big| |\chi_y(S)|^2$ for each $y \in G$. Since $G_p$ is cyclic, by Ma's Lemma \cite{Ma} (see also \cite[Chapter VI, Corollary 13.5]{BJL}), we have
\begin{equation}\label{eqn-respf1}
SS^{(-1)}=p^rX+PY,
\end{equation}
where $X,Y \in \Z[G]$ have nonnegative coefficients and $P$ is the unique cyclic subgroup of order $p$ in $G_p$. By \eqref{eqn-def2}, since $b_S \big| \nu_S(y)$ for each $y \in G$, we have
\begin{equation}\label{eqn-respf2}
SS^{(-1)}=|S|+b_SZ,
\end{equation}
where $Z \in \Z[G]$ with $[Z]_1=0$. Set $i=[X]_1$ and $j=[PY]_1$. Since $X$ and $Y$ have nonnegative coefficients, then $i$ and $j$ are two nonnegative integers satisfying $p^ri+j=[SS^{(-1)}]_1=|S|$ and therefore $0 \le i \le \lf \frac{|S|}{p^r} \rf$. For each $g \in P \sm \{1\}$, we have $[X]_g=i^{\pr}$ for some nonnegative integer $i^{\pr}$ and $[PY]_g=[PY]_1=j$. Therefore, $\nu_S(g)=[SS^{(-1)}]_g=p^ri^{\pr}+j=p^r(i^\pr-i)+|S|$. We claim that $0 \le i^\pr <i \le \lf \frac{|S|}{p^r} \rf$. Indeed, since $S$ and $T$ form a primitive formally dual pair and $g$ is nonidentity, by Lemma~\ref{lem-priset}, we have $\nu_S(g)=[SS^{(-1)}]_g < |S|$, which implies $i^\pr < i$.

By \eqref{eqn-respf2}, we have $b_S \big| \nu_S(g)$. Thus, $b_S \big| (p^r(i^\pr-i)+|S|)$ and therefore, $b_S \big| p^r(i-i^{\pr})$, where $0 \le i^\pr <i \le \lf \frac{|S|}{p^r} \rf$. Thus, $\frac{b_S}{(b_S,p^r)} \big| (i-i^{\pr})$ holds for some $0 \le i^\pr <i \le \lf \frac{|S|}{p^r} \rf$, which is equivalent to $\lf \frac{|S|}{p^r} \rf \ge \frac{b_S}{(b_S,p^r)}$.
\end{proof}

As a direct consequence of Proposition~\ref{prop-decom}, we have the following corollary.

\begin{corollary}\label{cor-proj}
Let $S$ and $T$ be a formally dual pair in $G$. Let $y \in G$ be an element of prime order $p$. Define $H_y=G/\ker \chi_y$. Let $\rho_y: G \rightarrow H_y$ be the natural projection. Then we have
$$
\rho_y(S)\rho_y(S)^{(-1)}=\frac{|S|^2}{|T|}\nu_T(y)+\frac{|S|^2}{p}(1-\frac{\nu_T(y)}{|T|})H_y
$$
and
$$
\rho_y(T)\rho_y(T)^{(-1)}=\frac{|T|^2}{|S|}\nu_S(y)+\frac{|T|^2}{p}(1-\frac{\nu_S(y)}{|S|})H_y.
$$
\end{corollary}
\begin{proof}
By Lemma~\ref{lem-char}(2), the character $\chi_y$ induces a character $\wti{\chi_y}$ of $H_y$ of order $p$, such that for each $g \in G$, we have $\wti{\chi_y}(\rho_y(g))=\chi_y(g)$. By Proposition~\ref{prop-decom}(1), for each $i \in \Zp^*$,
$$
|\wti{\chi_{y}}^i(\rho_y(S))|^2=|\wti{\chi_{y}}(\rho_y(S)^{(i)})|^2=|\wti{\chi_{y}}(\rho_y(S^{(i)}))|^2=|\chi_{y}(S^{(i)})|^2=|\chi_y(S)|^2=\frac{|S|^2}{|T|}\nu_T(y).
$$
By Lemma~\ref{lem-char}(2), the character $\wti{\chi_y}$ generates the character group $\wh{H_y} \cong \Zp$ and $\{\wti{\chi_y}^i \mid i \in \Zp^*\}$ is exactly the set of all nonprincipal characters of $H_y$.
Using the Fourier inversion formula, we have
$$
\rho_y(S)\rho_y(S)^{(-1)}=\frac{|S|^2}{|T|}\nu_T(y)+\frac{|S|^2}{p}(1-\frac{\nu_T(y)}{|T|})H_y.
$$
By interchanging the roles of $S$ and $T$, we complete the proof.
\end{proof}

The next theorem follows from Corollary~\ref{cor-proj}, which provides some restrictions on the weight enumerators of $S$ and $T$.

\begin{theorem}\label{thm-weight}
Let $S$ and $T$ be a primitive formally dual pair in a group $G$. Let $p$ be a prime divisor of $|G|$. Let $y \in G$ be an element of order $p$. Then
$$
\frac{|S|^2}{p}(1-\frac{\nu_T(y)}{|T|}), \quad \frac{|T|^2}{p}(1-\frac{\nu_S(y)}{|S|})
$$
are positive integers. In particular, if $p \nmid |S|$, we have $(|S|^2,|T|)>p$ and if $p \nmid |T|$, we have $(|T|^2,|S|)>p$.
\end{theorem}
\begin{proof}
Since $S$ and $T$ form a primitive formally dual pair, by Lemma~\ref{lem-pri}(1), for each $y \in G \sm\{1\}$, we have $\nu_S(y)<|S|$ and $\nu_T(y) < |T|$. Together with Corollary~\ref{cor-proj}, we know that
$$
\frac{|S|^2}{p}(1-\frac{\nu_T(y)}{|T|}), \quad \frac{|T|^2}{p}(1-\frac{\nu_S(y)}{|S|})
$$
are both positive integers.

Noting that $b_T \big| \nu_T(y)$, set $\nu_T(y)=b_Ti=\frac{|T|}{(|S|^2,|T|)}i$ for some $0 \le i < (|S|^2,|T|)$. Consequently,
$\frac{|S|^2}{p}(1-\frac{i}{(|S|^2,|T|)})$ is a positive integer for some $0 \le i < (|S|^2,|T|)$. If $p \nmid |S|$, we have $p \big| ((|S|^2,|T|)-i)$ for some $0 \le i < (|S|^2,|T|)$, which implies $(|S|^2,|T|)>p$. A similar argument shows that $(|T|^2,|S|)>p$ if $p \nmid |T|$.
\end{proof}

\begin{remark}
In Theorem~\ref{thm-weight}, if $G_p$ is cyclic, then the restrictions $(|S|^2,|T|)>p$ and $(|T|^2,|S|)>p$ are implied by Theorem~\ref{thm-order}.
\end{remark}

Next, we want to derive some restrictions on the character values $|\chi(S)|^2$ and $|\chi(T)|^2$. Let $a$ and $n$ be two positive integers. We say that $a$ is a \emph{primitive root} modulo $n$ if $a$ generates the unit group $\Z_n^*$. Let $n$ be a positive integer and $p$ be a prime divisor of $n$. Write $n = p^bn^\pr$, where $b \ge 0$ and $n^\pr$ is coprime with $p$. We say that $p$ is \emph{self-conjugate} modulo $n$, if there exists a nonnegative integer $j$, such that $p^j \equiv -1 \pmod{n^\pr}$. If a prime $p$ is a primitive root modulo $n$, then $p$ is self-conjugate modulo~$n$. For $X \in \Z[\ze_n]$, we use $(X)$ to denote the principal ideal generated by $X$ in $\Z[\ze_n]$. The following is a preparatory lemma.

\begin{lemma}\label{lem-chardiv}
Let $p$ and $q$ be primes and $q$ be a primitive root modulo $p^e$, where $e$ is a positive integer. Let $X \in \Z[\ze_{p^e}]$. Suppose $q^f \big\| |X|^2 $, then $f$ must be even. In particular, if $q^f \big| |X|^2$ and $f$ is odd, then $q^{f+1} \big| |X|^2$.
\end{lemma}
\begin{proof}
Since $q$ is a primitive root modulo $p^e$, we know that $(q)$ is a prime ideal in $\Z[\ze_{p^e}]$ \cite[Chapter 13, Theorem 2]{IR}, which we denote by $\cQ$. By $q^f \big\| |X|^2 $, we have $\cQ^f \big\|(X)(\ol{X})$. Suppose $\cQ^{f_1} \big\|(X)$ and $\cQ^{f_2} \big\|(\ol{X})$, where $f=f_1+f_2$. Since $q$ is self-conjugate modulo $p^e$, we have that $\cQ$ is invariant under complex conjugation \cite[Chapter VI, Corollary 15.5]{BJL}. Thus,  $\cQ^{f_1}=\ol{\cQ}^{f_1}$ and $\cQ^{f_1} \big\|(\ol{X})$, which implies $f_1=f_2$ and $f$ being even.
\end{proof}

Employing Lemma~\ref{lem-chardiv}, we obtain some information about the divisibility of character values $|\chi(S)|^2$ and $|\chi(T)|^2$.

\begin{proposition}\label{prop-chardiv}
Let $S$ and $T$ be a primitive formally dual pair in a group $G$. Then, for each $\chi \in \wh{G}$, we have $a_S \big| |\chi(S)|^2$ and $a_T \big| |\chi(T)|^2$. Let $p$ be a prime divisor of $|G|$. Let $H$ be a direct factor of $G_p$ with $H \cong \Z_{p^e}$. Moreover, let $\rho: G \rightarrow H$ be the natural projection and $q$ be a prime and a primitive root modulo $p^e$. Then we have the following.
\begin{itemize}
\item[(1)] Suppose $q^{f_1} \big| a_S$ with $f_1$ being odd. Then for each $\wti{\chi} \in \wh{H}$, we have $q^{f_1+1} \big| |\wti{\chi}(\rho(S))|^2$.
\item[(2)] Suppose $q^{f_2} \big| a_T$ with $f_2$ being odd. Then for each $\wti{\chi} \in \wh{H}$, we have $q^{f_2+1} \big| |\wti{\chi}(\rho(T))|^2$.
\end{itemize}
\end{proposition}
\begin{proof}
By \eqref{eqn-def}, we have $|\chi(S)|^2 \in \Z$, and hence, $a_S \big| |\chi(S)|^2$, for each $\chi \in \wh{G}$. Similarly, we can show that  $a_T \big| |\chi(T)|^2$, for each $\chi \in \wh{G}$.

We only prove (1), since the proof of (2) is analogous. Each $\chi \in \wh{G}$ induces a character $\wti{\chi} \in \wh{H}$, such that $\wti{\chi}(\rho(g))=\chi(g)$, for each $g \in G$. Consequently, we have $|\wti{\chi}(\rho(S))|^2=|\chi(S)|^2$ and $a_S \big| |\wti{\chi}(\rho(S))|^2$ for each $\wti{\chi} \in \wh{H}$. Since $q^{f_1} \big| |\wti{\chi}(\rho(S))|^2 \in \Z[\ze_{p^e}]$ and $q$ is a primitive root modulo $p^e$, by Lemma~\ref{lem-chardiv}, we have $q^{f_1+1} \big| |\wti{\chi}(\rho(S))|^2$, for each $\wti{\chi} \in \wh{H}$.
\end{proof}

The next example illustrates how to combine Theorem~\ref{thm-weight} and Proposition~\ref{prop-chardiv} to derive the nonexistence of some primitive formally dual pairs.

\begin{example}
We are going to show there exists no primitive formally dual pair $S$ and $T$ in $G=\Z_{180}$, with $|S|=6$ and $|T|=30$. Suppose otherwise, let $y \in G$ be a nonidentity element of order $5$ and $\chi_y \in \wh{G}$ be a character corresponding to $y$. By Theorem~\ref{thm-weight}, we have $\nu_T(y)=5$ and therefore $|\chi_y(S)|^2=6$. Let $\rho_y: G \rightarrow H$ be the natural projection, where $H=G/\ker \chi_y \cong \Z_5$ by Lemma~\ref{lem-char}(2). Note that $\chi_y$ induces a character $\wti{\chi_y}$ defined on $H$, such that $|\chi_y(S)|^2=|\wti{\chi_y}(\rho_y(S))|^2$. Note that $a_S=6$ and by Proposition~\ref{prop-chardiv}, we have $6 \big| |\wti{\chi_y}(\rho_y(S))|^2 \in \Z[\ze_5]$.  Since both $2$ and $3$ are primitive roots modulo $5$, by Proposition~\ref{prop-chardiv}(2), we have $2^2 \big| |\wti{\chi_y}(\rho_y(S))|^2$ and $3^2 \big| |\wti{\chi_y}(\rho_y(S))|^2$, which implies $36 \big| |\wti{\chi_y}(\rho_y(S))|^2$. Therefore, we have $36 \big| |\chi_y(S)|^2$, which contradicts $|\chi_y(S)|^2=6$.
\end{example}

Under the self-conjugate assumption, we have the following theorem, which extends Proposition~\ref{prop-cyclic}(10).

\begin{theorem}\label{thm-selfconj}
Let $G$ be a group so that a prime $p \big| |G|$ and $G_p$ is cyclic. Let $p$ be self-conjugate modulo $\exp(G)$. Suppose $S$ and $T$ form a primitive formally dual pair in $G$, with $p^{l_1} \big\| |S|$ and $p^{l_2} \big\| |T|$. Then $l_1=l_2=1$, and consequently, $p^2 \big\| |G|$.
\end{theorem}
\begin{proof}
Without loss of generality, we assume $l_1 \ge l_2 \ge 0$. By \eqref{eqn-def}, for each $y \in G$, we have $p^{2l_1-l_2} \big| |\chi_y(S)|^2$. It is easy to verify that except $l_1=l_2=1$, we have $p^2 \big| |\chi(S)|^2$ for each $\chi \in \wh{G}$. When $l_1=l_2=1$ does not hold, since $p$ is self-conjugate module $\exp(G)$, by \cite[Chapter VI, Lemma 13.2]{BJL}, $p \big| \chi(S)$ for each $\chi \in \wh{G}$. Note that $G_p$ is cyclic, by Ma's Lemma \cite{Ma}, we have
$$
S=pX+PY,
$$
where $X,Y \in \Z[G]$ and $P$ is the unique cyclic subgroup of order $p$ in $G_p$. Since $S$ is a set, then $X=0$. Thus $S=PY$. By Lemma~\ref{lem-pri}(2), $S$ and $T$ do not form a primitive formally dual pair.
\end{proof}

We observe that the TITO configuration satisfies the conditions in the above theorem, where $G=\Z_4$, $p=2$ and $l_1=l_2=1$.

We proceed to show that the structural information about the group $G$ provides a lower bound on the size of a primitive subset in $G$. The \emph{minimal number of generators of $G$}, is the smallest possible size of a subset in $G$, which generates $G$.

\begin{proposition}\label{prop-generator}
Let $S$ be a primitive subset in $G$. Suppose the minimal number of generators of $G$ is $s$. Then $|S| \ge s+1$.
\end{proposition}
\begin{proof}
With out loss of generality, we can assume that $1 \in S$. Since the minimal number of generators of $G$ is $s$, if $|S|<s+1$, then the elements of $S$ generate a proper subgroup of $G$. Therefore, $S$ is contained in a proper subgroup of $G$ and not a primitive subset.
\end{proof}

Below we present an exponent bound for the groups containing primitive formally dual sets.
\begin{proposition}\label{prop-projorder}
Let $G$ be a group and $p$ be a prime divisor of $|G|$. Suppose $\exp(G_p)=p^e$. Let $H$ be a cyclic subgroup of order $p^e$ and a direct factor of $G$. Suppose $\rho: G \rightarrow H$ is the natural projection and $S$ is an even set and a primitive subset in $G$. Then the following holds.
\begin{itemize}
\item[(1)] For any $y \in H$ and $i \in \Z_{p^e}^*$, we have $[\rho(S)\rho(S)^{(-1)}]_{y^i}=[\rho(S)\rho(S)^{(-1)}]_{y}$.
\item[(2)] $|S|^2-|S| \ge |S|^2-[\rho(S)\rho(S)^{(-1)}]_1 \ge p^{e-1}(p-1)$.
\end{itemize}
\end{proposition}
\begin{proof}
Since $S$ is an even set, we have $SS^{(-1)}=\sum_{i} \la_iL_i$, where $L_i$'s are subgroups of $G$. Applying the projection $\rho$ on both sides, we have
\begin{equation}\label{eqn-decom}
\rho(S)\rho(S)^{(-1)}=\sum_{i} \la_i \rho(L_i),
\end{equation}
where the $\rho(L_i)$'s are multisets formed by copies of subgroups of $H$. Suppose $h$ is a generator of $H$. For $0 \le i \le e$, define a subgroup $H_i=\lan h^{p^{e-i}}\ran$. Then $\{1\}=H_0 \lneqq H_1 \lneqq \cdots \lneqq H_e=H$ forms a chain containing all subgroups of $H$. Therefore, we can rewrite \eqref{eqn-decom} as
$$
\rho(S)\rho(S)^{(-1)}=\sum_{i=0}^e \mu_i H_i,
$$
where $\mu_i$, $0 \le i \le e$, are integers. Thus, for any $y \in H$, we know that $[\rho(S)\rho(S)^{(-1)}]_y$ only depends on the order of $y$. Consequently, the conclusion of (1) follows.

We claim the $\mu_e \ge 1$. Clearly, $\mu_e \ge 0$. Assume that $\mu_e=0$. There exists a nonprincipal character $\wti{\chi} \in \wh{H}$, which is principal on $H_{e-1}$. The character $\wti{\chi}$ induces a lifting character $\chi \in \wh{G}$, such that
$$
|\chi(S)|^2=|\wti{\chi}(\rho(S))|^2=|S|^2.
$$
By Lemma~\ref{lem-priset}, $S$ is not a primitive subset, which is a contradiction.

Applying (1), for any $i \in \Z_{p^e}^*$, we have $[\rho(S)\rho(S)^{(-1)}]_{h^i}=\mu_e$. Considering the number of nonidentity elements in $[\rho(S)\rho(S)^{(-1)}]$, we have the following exponent bound:
$$
|S|^2-|S| \ge |S|^2-[\rho(S)\rho(S)^{(-1)}]_1 \ge p^{e-1}(p-1)\mu_e \ge p^{e-1}(p-1).
$$
\end{proof}

Furthermore, for primitive formally dual sets whose size is a prime, we have a more precise description as follows.

\begin{corollary}\label{cor-projorder}
\item[(1)] A primitive formally dual set of size $2$ must be a $(2,2,2,1)$-RDS in $\Z_4$ relative to $2\Z_4$.
\item[(2)] Let $p$ be an odd prime. Suppose there exists a primitive formally dual set of size $p$ in $G$. Then $G_p$ is elementary abelian and $G_p$ contains a subgroup isomorphic to $\Zp \times \Zp$.
\end{corollary}
\begin{proof}
(1) Let $S$ be a primitive formally dual set of size $2$ in $G$. By Proposition~\ref{prop-generator}, the group $G$ has only one generator and is a cyclic group $\Z_N$. Let $p$ be a prime divisor of $N$ and suppose $p^f \big\| N$ for some integer $f$. Then $\exp(G_p)=p^f$. By Proposition~\ref{prop-projorder}(2), we have
$$
2=|S|^2-|S| \ge p^{f-1}(p-1),
$$
which implies $p=2$ and $f \in \{1,2\}$. If $f=1$, then $G=\Z_2$ and $S=G$ is not a primitive subset of $G$. The only possible case is $p=2$, $f=2$ and $G=\Z_4$, in which $S$ must be a $(2,2,2,1)$-RDS in $\Z_4$ relative to $2\Z_4$, described in Proposition~\ref{prop-cycRDS}(2).

(2) Let $S$ be a primitive formally dual set of size $p$ in $G$, where $\exp(G_p)=p^e$ for some $e \ge 1$. Then $G$ contains a direct factor $H$ isomorphic to $\Z_{p^e}$. Let $\rho: G \rightarrow H$ be the natural projection. By Proposition~\ref{prop-projorder}(2), we have
$$
p^2-p=|S|^2-|S| \ge |S|^2-[\rho(S)\rho(S)^{(-1)}]_1 \ge p^{e-1}(p-1),
$$
which forces $e \in \{1,2\}$. If $e=2$, let $h$ be a generator of $H \cong \Z_{p^2}$. Note that the equalities hold in the above expression and $[\rho(S)\rho(S)^{(-1)}]_1=|S|$. This implies that $\rho(S)$ is indeed a subset of $H$. Since $S$ is a primitive subset, we can easily see that $\{\rho(S)\rho(S)^{(-1)}\}$ is not contained in the subgroup $\lan h^p \ran$. Thus, there exists some $y \in H \sm \lan h^p \ran$, so that $[\rho(S)\rho(S)^{(-1)}]_{y} \ge 1$. By Proposition~\ref{prop-projorder}(1), for each $i \in \Z_{p^2}^*$, we have $[\rho(S)\rho(S)^{(-1)}]_{y^i}=[\rho(S)\rho(S)^{(-1)}]_{y}$. Therefore, the $p^2-p$ nonidentity elements of $[\rho(S)\rho(S)^{(-1)}]$ are exactly the elements of $H \sm \lan h^p \ran$. So, $\rho(S)$ is a $(p,p,p,1)$-RDS in the cyclic group $H \cong Z_{p^2}$ relative to $\lan h^p \ran \cong \Zp$. Since $p$ is an odd prime, this contradicts Proposition~\ref{prop-cycRDS}.

Consequently, we must have $e=1$ and $G_p$ is elementary abelian. Suppose $S$ and $T$ form a primitive formally dual pair in $G$. By Proposition~\ref{prop-non}(2), $p \big| |T|$. By Proposition~\ref{prop-non}(1), $p^2 \big| |G|$. Hence, $G_p$ contains a subgroup isomorphic to $\Zp \times \Zp$.
\end{proof}

\begin{remark}\label{rem-cycopen}
Combining Proposition~\ref{prop-cyclic}(1)(3)(4), Proposition~\ref{prop-non}, Theorems~\ref{thm-order}, \ref{thm-weight}, \ref{thm-selfconj} and Proposition~\ref{prop-chardiv}, we derive that there are three open cases of primitive formally dual pairs in cyclic groups $\ZN$, where $N \le 1000$ and $|S| \le |T|$:
$$
(N,|S|,|T|) \in \{(600,10,60),(784,28,28),(900,30,30)\}.
$$
\end{remark}

\begin{remark}\label{rem-search}
For every abelian group of order up to $63$, we can either, up to equivalence, list all primitive formally dual sets in the group, or establish a nonexistence result, see Table~\ref{tab-table} in Appendix~\ref{app-A}. This classification is achieved by exploiting Proposition~\ref{prop-cyclic}(1)(3), Proposition~\ref{prop-non}(1)(2), Theorems~\ref{thm-order}, \ref{thm-weight}, \ref{thm-selfconj}, Proposition~\ref{prop-generator} and Corollary~\ref{cor-projorder}, together with a computer search. We note that the computer search for the next interesting case, the abelian groups of order $64$, is very time-consuming.
\end{remark}

\section{Concluding remarks}\label{sec6}

Formal duality was proposed in \cite{CKS}, which reflects a deep symmetry among energy-minimizing periodic configurations. Following the idea in \cite{CKRS}, formal duality between two periodic configurations can be translated into a combinatorial setting, namely, a formally dual pair in a finite abelian group. In this paper, we initiated a systematic investigation on formally dual pairs in finite abelian groups. For primitive formally dual pairs in finite abelian groups, we derived some nonexistence results. They are in favor of the main conjecture, which claims there are only two small examples of primitive formally dual pairs in cyclic groups. On the other hand, for primitive formally dual pairs in noncyclic groups, we gave three constructions living in elementary abelian groups or products of Galois rings. A reflection on our constructions motivated us to propose the concept of even sets, which provides more structural information about formally dual pairs. Moreover, the even set viewpoint led to a characterization of rank three primitive formally dual pairs, which sheds some new light on the main conjecture.

The numerical simulations in \cite{CKS} suggested that for each energy-minimizing periodic configuration, one can find another periodic configuration as its dual. This pair of periodic configurations generates a formally dual pair in a finite abelian group. The formally dual pair indicates how one can form the periodic configurations by taking the union of translations of a given lattice. In this sense, our new constructions of formally dual pairs produce potential schemes to form energy-minimizing periodic configurations.

We note that there are still a lot of problems about formally dual pairs which deserve further investigation and we mention a few below.
\begin{itemize}
\item[(1)] Construct new primitive formally dual pairs which are inequivalent to the known ones. Our constructions suggest there might be more in noncyclic groups.
\item[(2)] Note that for each known $(n,n,n,1)$-RDS in $G$ relative to $N$, we have $\wti{N} \cong N$. Therefore, for all known examples, the converse of Theorem~\ref{thm-rankthree} is also true. Is there any $(n,n,n,1)$-RDS in $G$ relative to $N$, such that $\wti{N} \not\cong N$? We expect that this is not possible.
\item[(3)] As noted in Remark~\ref{rem-semi}, there are many different skew Hadamard difference sets available, which give distinct primitive formally dual pairs. Are these formally dual pairs equivalent or not?
\item[(4)] Derive more nonexistence results about primitive formally dual pairs in cyclic groups. In particular, is there any way to tackle the three open cases in Remark~\ref{rem-cycopen}?
\item[(5)] According to Table~\ref{tab-table} in Appendix~\ref{app-A}, the nonexistence of primitive formally dual pairs in finite abelian groups remains widely open. Is there any theoretic approach to deal with the examples in Table~\ref{tab-table}, whose nonexistence is confirmed only by computer search?
\end{itemize}

\section*{Appendix}

\appendix

\section{Classification of primitive formally dual sets in abelian groups of order up to 63}\label{app-A}

In this appendix, we provide a table of all inequivalent examples and nonexistence results about primitive formally dual sets in abelian group of order up to $63$. If a group does contain a primitive formally dual set, we list all inequivalent examples obtained by exhaustive computer search and identify the construction approach which explains the example. If a group does not contain a primitive formally dual set, we either list the theoretical explanation or indicate that the nonexistence follows from exhaustive computer search. By Proposition~\ref{prop-non}(2), we only need to consider the groups whose order are not square-free. By Proposition~\ref{prop-non}(1), we only focus on subset $S$ of group $G$, such that $|S| \big| |G|$. By the symmetry between the two subsets in a primitive formally dual pair, we only list the primitive formally dual set $S$ in $G$, with $|S| \le \sqrt{|G|}$. Note that every rank three primitive formally dual set $S$ in $G$, satisfying $2 \le |S| \le 6$ and $|G| \le |S|^2$, is marked by $\bigstar$ in Table~\ref{tab-table}.


\ra{1.1}
\begin{longtable}{|c|c|c|c|}
\caption{All inequivalent examples and nonexistence of primitive formally dual sets in abelian groups of order up to $63$}
\label{tab-table}
\\ \hline
$(|G|,|S|)$ & $G$ & Primitive formally dual sets & Source \\ \hline
(1,1) & $\{1\}$ & \{1\} & Trivial \\ \hline
\multirow{2}{*}{(4,2)} & $\Z_4$ & (2,2,2,1)-RDS $\bigstar$ & Theorem~\ref{thm-RDS}(1) \\
                 & $\Z_2^2$ & none & Corollary~\ref{cor-projorder}(1) \\ \hline
(8,2) & arbitrary & none & Corollary~\ref{cor-projorder}(1) \\ \hline
\multirow{2}{*}{(9,3)} & $\Z_9$ & none & Corollary~\ref{cor-projorder}(2)\\
                       & $\Z_3^2$ & (3,3,3,1)-RDS $\bigstar$ & Theorem~\ref{thm-RDS}(2)\\ \hline
(12,2) & arbitrary & none & Corollary~\ref{cor-projorder}(1) \\ \hline
(12,3) & arbitrary & none & Proposition~\ref{prop-non}(2) \\ \hline
(16,2) & arbitrary & none & Corollary~\ref{cor-projorder}(1) \\ \hline
\multirow{6}{*}{(16,4)} & $\Z_{16}$ & none & Proposition~\ref{prop-cyclic}(1)\\
                        & $\Z_{2} \times \Z_8$ & none & computer search\\
                        & \multirow{2}{*}{$\Z_{4}^2$} & product of two $(2,2,2,1)$-RDSs &  Proposition~\ref{prop-prod} \\
                        &                             & $(4,4,4,1)$-RDS $\bigstar$ & Theorem~\ref{thm-RDS}(1) \\
                        & $\Z_{2}^2 \times \Z_4$ & none & computer search\\
                        & $\Z_2^4$ & none & Proposition~\ref{prop-generator}\\ \hline
(18,2) & arbitrary & none & Corollary~\ref{cor-projorder}(1) \\ \hline
\multirow{2}{*}{(18,3)} & $\Z_{2} \times \Z_9$ & none & Corollary~\ref{cor-projorder}(2) \\
                        & $\Z_{2} \times \Z_3^2$ & none & Theorem~\ref{thm-selfconj} \\\hline
(20,2) & arbitrary & none & Corollary~\ref{cor-projorder}(1) \\ \hline
(20,4) & arbitrary & none & Proposition~\ref{prop-non}(2) \\ \hline
(24,2) & arbitrary & none & Corollary~\ref{cor-projorder}(1) \\ \hline
(24,3) & arbitrary & none & Proposition~\ref{prop-non}(2) \\ \hline
\multirow{3}{*}{(24,4)} & $\Z_3 \times \Z_8$ & none & Theorem~\ref{thm-selfconj} \\
                        & $\Z_3 \times \Z_2 \times \Z_4$ & none & Theorem~\ref{thm-selfconj} \\
                        & $\Z_3 \times \Z_2^3$ & none & Theorem~\ref{thm-selfconj} \\ \hline
\multirow{2}{*}{(25,5)} & $\Z_{25}$ & none & Corollary~\ref{cor-projorder}(2)\\
                        & $\Z_{5}^2$ & (5,5,5,1)-RDS $\bigstar$ & Theorem~\ref{thm-RDS}(2)\\ \hline
\multirow{3}{*}{(27,3)} & $\Z_{27}$ & none & Corollary~\ref{cor-projorder}(2)\\
                        & $\Z_{3} \times \Z_9$ & none & Corollary~\ref{cor-projorder}(2)\\
                        & $\Z_{3}^3$ & none & Proposition~\ref{prop-generator}\\ \hline
(28,2) & arbitrary & none & Corollary~\ref{cor-projorder}(1) \\ \hline
(28,4) & arbitrary & none & Proposition~\ref{prop-non}(2) \\ \hline
(32,2) & arbitrary & none & Corollary~\ref{cor-projorder}(1) \\ \hline
\multirow{7}{*}{(32,4)} & $\Z_{32}$ & none & Proposition~\ref{prop-cyclic}(1)\\
                        & $\Z_{2} \times \Z_{16}$ & none & computer search\\
                        & $\Z_{4} \times \Z_{8}$ & none & computer search\\
                        & $\Z_{2}^2 \times \Z_{8}$ & none & computer search\\
                        & $\Z_{2} \times \Z_{4}^2$ & $\{(0,0,0), (0,0,1),(0,1,0), (1,1,1)\}$  & Example~\ref{exam-244}\\
                        & $\Z_{2}^3 \times \Z_4$ & none & Proposition~\ref{prop-generator}\\
                        & $\Z_2^5$ & none & Proposition~\ref{prop-generator}\\ \hline
(36,2) & arbitrary & none & Corollary~\ref{cor-projorder}(1) \\ \hline
\multirow{5}{*}{(36,3)} & $\Z_{4} \times \Z_{9}$ & none & Corollary~\ref{cor-projorder}(2)\\
                        & $\Z_{2}^2 \times \Z_{9}$ & none & Corollary~\ref{cor-projorder}(2)\\
                        & $\Z_{4} \times \Z_{3}^2$ & none & computer search \\
                        & $\Z_{2}^2 \times \Z_{3}^2$ & none & computer search \\ \hline
(36,4) & arbitrary & none & Proposition~\ref{prop-non}(2) \\ \hline
\multirow{5}{*}{(36,6)} & $\Z_{4} \times \Z_{9}$ & none & Proposition~\ref{prop-cyclic}(3)\\
                        & $\Z_{2}^2 \times \Z_{9}$ & none & computer search\\
                        & \multirow{2}{*}{$\Z_{4} \times \Z_{3}^2$} & product of a (2,2,2,1)-RDS & \multirow{2}{*}{Proposition~\ref{prop-prod}} \\
                        &                                           & and a (3,3,3,1)-RDS &  \\
                        & $\Z_{2}^2 \times \Z_{3}^2$ & none & computer search \\ \hline
(40,2) & arbitrary & none & Corollary~\ref{cor-projorder}(1) \\ \hline
\multirow{3}{*}{(40,4)} & $\Z_{8}\times \Z_5$ & none & Theorem~\ref{thm-selfconj}\\
                        & $\Z_{2} \times \Z_4 \times \Z_{5}$ & none & Theorem~\ref{thm-weight} \\
                        & $\Z_{2}^3 \times \Z_{5}$ & none & Theorem~\ref{thm-weight} \\  \hline
(40,5) & arbitrary & none & Proposition~\ref{prop-non}(2) \\ \hline
(44,2) & arbitrary & none & Corollary~\ref{cor-projorder}(1) \\ \hline
(44,4) & arbitrary & none & Proposition~\ref{prop-non}(2) \\ \hline
\multirow{2}{*}{(45,3)} & $\Z_{9}\times \Z_5$ & none & Corollary~\ref{cor-projorder}(2)\\
                        & $\Z_{3}^2 \times \Z_5$ & none & Theorem~\ref{thm-selfconj} \\  \hline
(45,5) & arbitrary & none & Proposition~\ref{prop-non}(2) \\ \hline
(48,2) & arbitrary & none & Corollary~\ref{cor-projorder}(1) \\ \hline
(48,3) & arbitrary & none & Proposition~\ref{prop-non}(2) \\ \hline
\multirow{5}{*}{(48,4)} & $\Z_{16}\times \Z_3$ & none & Proposition~\ref{prop-cyclic}(2)\\
                        & $\Z_{2}\times \Z_8 \times \Z_3$ & none & computer search  \\
                        & $\Z_{4}^2\times \Z_3$ & none & Theorem~\ref{thm-selfconj} \\
                        & $\Z_{2}^2\times \Z_4 \times \Z_3$ & none & Theorem~\ref{thm-selfconj} \\
                        & $\Z_{2}^4\times \Z_3$ & none & Theorem~\ref{thm-selfconj} \\ \hline
\multirow{5}{*}{(48,6)} & $\Z_{16}\times \Z_3$ & none & Proposition~\ref{prop-cyclic}(2)\\
                        & $\Z_{2}\times \Z_8 \times \Z_3$ & none & computer search \\
                        & $\Z_{4}^2\times \Z_3$ & none & Theorem~\ref{thm-selfconj} \\
                        & $\Z_{2}^2\times \Z_4 \times \Z_3$ & none & Theorem~\ref{thm-selfconj} \\
                        & $\Z_{2}^4\times \Z_3$ & none & Theorem~\ref{thm-selfconj} \\  \hline
\multirow{4}{*}{(49,7)} & $\Z_{49}$ & none & Proposition~\ref{prop-cyclic}(1)\\
                        & \multirow{3}{*}{$\Z_{7}\times \Z_7$} & (7,7,7,1)-RDS & Theorem~\ref{thm-RDS}(2)  \\
                        &                                      & formally dual set derived from & \multirow{2}{*}{Theorem~\ref{thm-SHDS}} \\
                        &                                      & skew Hadamard difference set &         \\  \hline
(50,2) & arbitrary & none & Corollary~\ref{cor-projorder}(1) \\ \hline
\multirow{2}{*}{(50,5)} & $\Z_{2} \times \Z_{25}$ & none & Proposition~\ref{prop-cyclic}(2)\\
                        & $\Z_{2} \times \Z_{5}^2$ & none & Theorem~\ref{thm-selfconj} \\  \hline
(52,2) & arbitrary & none & Corollary~\ref{cor-projorder}(1) \\ \hline
(52,4) & arbitrary & none & Proposition~\ref{prop-non}(2) \\ \hline
(54,2) & arbitrary & none & Corollary~\ref{cor-projorder}(1) \\ \hline
\multirow{3}{*}{(54,3)} & $\Z_{2} \times \Z_{27}$ & none & Proposition~\ref{prop-projorder}(2)\\
                        & $\Z_{2} \times \Z_3 \times \Z_{9}$ & none & Theorem~\ref{thm-selfconj} \\
                        & $\Z_{2} \times \Z_3^3$ & none & Proposition~\ref{prop-generator} \\ \hline
(54,6) & arbitrary & none & Theorem~\ref{thm-selfconj} \\ \hline
(56,2) & arbitrary & none & Corollary~\ref{cor-projorder}(1) \\ \hline
(56,4) & arbitrary & none & Theorem~\ref{thm-selfconj} \\ \hline
(56,7) & arbitrary & none & Proposition~\ref{prop-non}(2) \\ \hline
(60,2) & arbitrary & none & Corollary~\ref{cor-projorder}(1) \\ \hline
(60,3) & arbitrary & none & Proposition~\ref{prop-non}(2) \\ \hline
(60,4) & arbitrary & none & Proposition~\ref{prop-non}(2) \\ \hline
(60,5) & arbitrary & none & Proposition~\ref{prop-non}(2) \\ \hline
(60,6) & arbitrary & none & Theorem~\ref{thm-weight} \\ \hline
(63,3) & arbitrary & none & Theorem~\ref{thm-weight} \\ \hline
(63,7) & arbitrary & none & Proposition~\ref{prop-non}(2) \\ \hline
\end{longtable}

\section{Proof of Theorem~\ref{thm-rankthree}}\label{app-B}

In this appendix, we show that all rank three primitive formally dual sets $S$ in $G$, satisfying $3 \le |S| \le 6$ and $|G| \le |S|^2$, must be an $(n,n,n,1)$-RDS in $G$. This completes the proof of Theorem~\ref{thm-rankthree}. Let $H,K$ be two subgroups of $G$, then we use $\Span\{H,K\}$ to denote the subgroup of $G$ generated by $H$ and $K$.

\begin{proof}[Proof of Theorem~\ref{thm-rankthree} (Continued)]
Now we continue to finish the proof of Case III in Theorem~\ref{thm-rankthree}. Recall that $|S|=\la_1+\la_2+\la_3$, with $1 \le \la_i \le 2$, $1 \le i \le 3$. Define $c_1=|S_1|$, $c_2=|S_2|$ and $c_{12}=|S_{12}|$. Together with $S_1, S_2 \ne \es$, we have $c_1, c_2 \ge 1$ and $1+c_1+c_2+c_{12}=|S|$.

We split our discussion into two subcases.

Case III(a): $S_{12}=\es$. In this case, we have $S=1+S_1+S_2$ and
\begin{equation}\label{eqn-expSSinv}
SS^{(-1)}=1+S_1+S_1^{(-1)}+S_2+S_2^{(-1)}+S_1S_1^{(-1)}+S_2S_2^{(-1)}+S_1S_2^{(-1)}+S_2S_1^{(-1)}.
\end{equation}
Since $c_{12}=0$, we have $2 \le c_1+c_2 \le 5$. We claim that $H_1 \cap H_2=\{1\}$. Otherwise, for $w \in (H_1 \cap H_2) \sm \{1\}$, by \eqref{eqn-SIprop2}, we have $[SS^{(-1)}]_w=[S_{12}+S_{12}^{(-1)}+S_{12}S_{12}^{(-1)}]_w=0$, which contradicts \eqref{eqn-rankthree}. Therefore, by \eqref{eqn-SIprop2},
\begin{equation}\label{eqn-SIprop3}
\begin{aligned}
&\{S_1+S_1^{(-1)}\} \subset H_1 \sm H_3, \quad \{S_2+S_2^{(-1)}\} \subset H_2 \sm H_3,\\
&\{S_1S_1^{(-1)}\} \subset H_1 \cap H_3, \quad \{S_2S_2^{(-1)}\} \subset H_2 \cap H_3, \\
&\{S_1S_2^{(-1)}+S_2S_1^{(-1)}\} \subset H_3 \sm (H_1 \cup H_2).
\end{aligned}
\end{equation}
We split our discussion into the following three subcases.

Case III(a1): $H_1 \cap H_3 \ne \{1\}$ and $H_2 \cap H_3 \ne \{1\}$. Combining \eqref{eqn-rankthree} and \eqref{eqn-SIprop3}, we have
$$
\la_1+\la_3=\frac{c_1^2-c_1}{|H_1 \cap H_3|-1}, \la_2+\la_3=\frac{c_2^2-c_2}{|H_2 \cap H_3|-1}.
$$
Since $\la_1+\la_3, \la_2+\la_3 \ge 2$, we have $c_1,c_2 \ge 2$ and $|S|=1+c_1+c_2 \ge 5$. Since $c_1+c_2 \le 5$, we can assume without loss of generality that $c_1=2$, which implies $\la_1+\la_3=\frac{2}{|H_1 \cap H_3|-1}=2$. Noting that $1 \le \la_2 \le 2$ and $|S|=\la_1+\la_2+\la_3 \le 4$, we derive a contradiction on the size of $S$.

Case III(a2): Exactly one of $H_1 \cap H_3$ and $H_2 \cap H_3$ is $\{1\}$, we assume without loss of generality that $H_1 \cap H_3=\{1\}$. Since $\{S_1S_1^{(-1)}\} \subset H_1 \cap H_3=\{1\}$, we have $c_1=|S_1|=1$. By \eqref{eqn-expSSinv} and \eqref{eqn-SIprop3}, we have
$$
SS^{(-1)}=2+S_1+S_1^{(-1)}+S_2+S_2^{(-1)}+S_2S_2^{(-1)}+S_1S_2^{(-1)}+S_2S_1^{(-1)},
$$
where
\begin{align*}
&\{S_1+S_1^{(-1)}\} \subset H_1 \sm \{1\}, \quad \{S_2+S_2^{(-1)}\} \subset H_2 \sm H_3, \\
&\{S_2S_2^{(-1)}\} \subset H_2 \cap H_3, \quad \{S_1S_2^{(-1)}+S_2S_1^{(-1)}\} \subset H_3 \sm H_2.
\end{align*}
Together with \eqref{eqn-rankthree}, we have
\begin{equation}\label{eqn-la}
\la_1=\frac{2}{|H_1 \sm \{1\}|}, \la_2=\frac{2c_2}{|H_2 \sm H_3|}, \la_3=\frac{2c_2}{|H_3 \sm H_2|}, \la_2+\la_3=\frac{c_2^2-c_2}{|H_2 \cap H_3|-1}.
\end{equation}
Recall that $c_1=1$, which implies $2 \le c_2 \le 4$.

If $c_2=2$, then $|S|=1+c_1+c_2=\la_1+\la_2+\la_3=4$ and $|G| \le 16$. Using \eqref{eqn-la},
\begin{align*}
\la_2+\la_3=\frac{2}{|H_2 \cap H_3|-1} &\Rightarrow |H_2 \cap H_3|=2, \la_2=\la_3=1, \\
\la_2=\frac{4}{|H_2 \sm H_3|}=1 &\Rightarrow |H_2 \sm H_3|=4, |H_2|=6, \\
\la_3=\frac{4}{|H_3 \sm H_2|}=1 &\Rightarrow |H_3 \sm H_2|=4, |H_3|=6.
\end{align*}
Therefore, $|\Span\{ H_2,H_3\}|=\frac{|H_2||H_3|}{|H_2 \cap H_3|}=18$, which implies $18 \big| |G|$. We derive a contradiction on the size of $G$.

If $c_2=3$, then $|S|=1+c_1+c_2=5$ and $|G| \le 25$. Together with Corollary~\ref{cor-projorder}(2), we have $G \cong \Z_5 \times \Z_5$. Note that $1 \le \la_1 \le 2$. By \eqref{eqn-la}, we have $|H_1| \in \{2,3\}$, which leads to a contradiction since $H_1$ is a subgroup of $G$.

If $c_2=4$, then $|S|=1+c_1+c_2=\la_1+\la_2+\la_3=6$, $\la_1=\la_2=\la_3=2$ and $|G| \le 36$. Using \eqref{eqn-la},
\begin{align*}
\la_2+\la_3=\frac{12}{|H_2 \cap H_3|-1}=4 &\Rightarrow |H_2 \cap H_3|=4, \\
\la_2=\frac{8}{|H_2 \sm H_3|}=2 &\Rightarrow |H_2 \sm H_3|=4, |H_2|=8, \\
\la_3=\frac{8}{|H_3 \sm H_2|}=2 &\Rightarrow |H_3 \sm H_2|=4, |H_3|=8.
\end{align*}
Since $|\Span\{H_2, H_3\}|=\frac{|H_2||H_3|}{|H_2 \cap H_3|}=16$, $16 \big| |G|$. Together with $|S| \big| |G|$, we have $48 \big| |G|$. We derive a contradiction on the size of $G$.

Case III(a3): $H_1 \cap H_3=\{1\}$ and $H_2 \cap H_3=\{1\}$. Similar to Case III(a2), we can show that $c_1=c_2=1$. Therefore, we get $|S|=3$ and $|G|\le 9$. Together with Corollary~\ref{cor-projorder}(2), we have $G \cong \Z_3 \times \Z_3$. Note that the action of affine group $\AGL(2,3)$ on the elements of $\Z_3 \times \Z_3$ is $2$-transitive \cite[Exercise 2.8.13]{DM}. We can assume after a proper transformation of $S$, we get $S^{\pr}=\{(1,0),(0,1),(z_1,z_2)\}$, where
$$
(z_1,z_2) \in \{(0,0),(1,0),(0,1),(1,1),(1,2),(2,1),(2,2)\}.
$$
If $(z_1,z_2)=(2,2)$, then $S^{\pr}S^{\pr(-1)}=3((0,0)+(1,2)+(2,1))$. Therefore, $SS^{(-1)}$ equals three copies of an order three subgroup of $\Z_3 \times \Z_3$, which means $S$ is not primitive by Lemma~\ref{lem-priset}. In the remaining cases, it is easy to verify that $S^{\pr}$, and hence $S$, is a $(3,3,3,1)$-RDS in $G$.

Case III(b): $S_{12} \ne \es$. In this case, we have $S=1+S_1+S_2+S_{12}$. Since $c_1, c_2, c_{12} \ge 1$, we have $4 \le |S|=1+c_1+c_2+c_{12}=\la_1+\la_2+\la_3 \le 6$ and $c_{12} \in \{1,2,3\}$. Since $\{S_{12}+S_{12}^{(-1)}\} \subset (H_1 \cap H_2) \sm \{1\}$ and $\{S_1S_2^{(-1)}+S_2S_1^{(-1)}\} \subset H_3 \sm (H_1 \cup H_2)$, we have $H_1 \cap H_2 \ne \{1\}$ and $H_3 \sm (H_1 \cup H_2) \ne \es$. Therefore,
\begin{align*}
SS^{(-1)}=&1+S_1+S_1^{(-1)}+S_2+S_2^{(-1)}+S_{12}+S_{12}^{(-1)}+S_1S_1^{(-1)}+S_2S_2^{(-1)}+S_{12}S_{12}^{(-1)}\\
          &+S_1S_2^{(-1)}+S_2S_1^{(-1)}+S_1S_{12}^{(-1)}+S_{12}S_1^{(-1)}+S_2S_{12}^{(-1)}+S_{12}S_2^{(-1)}.
\end{align*}
We restate \eqref{eqn-SIprop2} as follows.
\begin{equation}\label{eqn-SIprop4}
\begin{aligned}
& \{S_1+S_1^{(-1)}\} \subset H_1 \sm (H_2 \cup H_3), \quad \{S_2+S_2^{(-1)}\} \subset H_2 \sm (H_1 \cup H_3)\\
&\{S_{12}+S_{12}^{(-1)}\} \subset (H_1 \cap H_2) \sm \{1\}, \quad \{S_1S_1^{(-1)}\} \subset H_1 \cap H_3, \quad \{S_2S_2^{(-1)}\} \subset H_2 \cap H_3 \\
&\{S_{12}S_{12}^{(-1)}\} \subset H_1 \cap H_2, \quad \{S_1S_2^{(-1)}+S_2S_1^{(-1)}\} \subset H_3 \sm (H_1 \cup H_2) \\
&\{S_1S_{12}^{(-1)}+S_{12}S_1^{(-1)}\} \subset H_1 \sm H_2=((H_1 \cap H_3) \sm \{1\}) \cup (H_1 \sm (H_2 \cup H_3)) \\
&\{S_2S_{12}^{(-1)}+S_{12}S_2^{(-1)}\} \subset H_2 \sm H_1=((H_2 \cap H_3) \sm \{1\}) \cup (H_2 \sm (H_1 \cup H_3))
\end{aligned}
\end{equation}
According to \eqref{eqn-SIprop4}, we make the following observations.

Noting that $[SS^{(-1)}] \cap ((H_1 \cap H_2) \sm \{1\})=[S_{12}+S_{12}^{(-1)}+S_{12}S_{12}^{(-1)}] \cap ((H_1 \cap H_2) \sm \{1\})$, we have
\begin{equation}\label{eqn-la1la2}
  \la_1+\la_2=\frac{c_{12}^2+c_{12}}{|(H_1 \cap H_2) \sm \{1\}|}.
\end{equation}

Noting that $[SS^{(-1)}] \cap (H_3 \sm (H_1 \cup H_2))=[S_1S_{2}^{(-1)}+S_{2}S_{1}^{(-1)}] \cap (H_3 \sm (H_1 \cup H_2))$, we have
\begin{equation}\label{eqn-la3}
  \la_3=\frac{2c_1c_2}{|H_3 \sm (H_1 \cup H_2)|}.
\end{equation}

If $c_1=1$, then $[SS^{(-1)}] \cap ((H_1 \cap H_3) \sm \{1\})=[S_1S_{12}^{(-1)}+S_{12}S_1^{(-1)}] \cap ((H_1 \cap H_3) \sm \{1\})$. Thus,
\begin{equation}\label{eqn-c1}
  2c_1c_{12} \ge (\la_1+\la_3)|(H_1 \cap H_3) \sm \{1\}|.
\end{equation}
Similarly, if $c_2=1$, we have
\begin{equation}\label{eqn-c2}
  2c_2c_{12} \ge (\la_2+\la_3)|(H_2 \cap H_3) \sm \{1\}|.
\end{equation}

If $H_1 \cap H_3 =\{1\}$, then $[SS^{(-1)}] \cap (H_1 \sm (H_2 \cup H_3))=[S_1+S_1^{(-1)}+S_1S_{12}^{(-1)}+S_{12}S_1^{(-1)}] \cap (H_1 \sm (H_2 \cup H_3))$. Thus,
\begin{equation}\label{eqn-la1}
  \la_1=\frac{2c_1c_{12}+2c_1}{|H_1 \sm (H_2 \cup H_3)|}.
\end{equation}
Similarly, if $H_2 \cap H_3 =\{1\}$, we have
\begin{equation}\label{eqn-la2}
  \la_2=\frac{2c_2c_{12}+2c_2}{|H_2 \sm (H_1 \cup H_3)|}.
\end{equation}
Recalling that $c_{12} \in \{1,2,3\}$, we split our discussion into three cases.

Case III(b1): $c_{12}=1$. By \eqref{eqn-la1la2},
$$
 \la_1+\la_2=\frac{2}{|(H_1 \cap H_2) \sm \{1\}|} \Rightarrow \la_1=\la_2=1, |H_1 \cap H_2|=2.
$$
Noting that $4 \le 1+c_1+c_2+c_{12}=|S|=\la_1+\la_2+\la_3 \le 4$, we have $c_1=c_2=1$, $\la_3=2$, $|S|=4$ and $|G| \le 16$. We claim $H_1 \cap H_3=\{1\}$. Otherwise, since $c_1=1$, by \eqref{eqn-c1},
$$
2 \ge 3|(H_1 \cap H_3) \sm \{1\}| \ge 3,
$$
which is a contradiction. Similarly, we can show that $H_2 \cap H_3=\{1\}$.

Since $H_1 \cap H_3=\{1\}$, $H_2 \cap H_3=\{1\}$ and $|H_1 \cap H_2|=2$, by \eqref{eqn-la1} and \eqref{eqn-la2},
\begin{align*}
\la_1=\frac{4}{|H_1 \sm (H_2 \cup H_3)|}=1 &\Rightarrow |H_1 \sm (H_2 \cup H_3)|=4, |H_1|=6, \\
\la_2=\frac{4}{|H_2 \sm (H_1 \cup H_3)|}=1 &\Rightarrow |H_2 \sm (H_1 \cup H_3)|=4, |H_2|=6.
\end{align*}
Therefore, $|\Span\{H_1,H_2\}|=18$ and $18 \big| |G|$. We derive a contradiction on the size of $G$.

Case III(b2): $c_{12}=2$. Since $|S|=1+c_1+c_2+c_{12}=\la_1+\la_2+\la_3 \ge 5$ and $1 \le \la_3 \le 2$. we have $\la_1+\la_2 \ge 3$. By \eqref{eqn-la1la2},
$$
\la_1+\la_2=\frac{6}{|(H_1 \cap H_2) \sm \{1\}|} \Rightarrow \la_1+\la_2=3, |H_1 \cap H_2|=3, \la_3=2, c_1=c_2=1, |S|=5, |G|\le 25.
$$
Together with Corollary~\ref{cor-projorder}(2), we have $G \cong \Z_5 \times \Z_5$, which contradicts $|H_1 \cap H_2|=3$.

Case III(b3): $c_{12}=3$. Since $|S|=\la_1+\la_2+\la_3=1+c_1+c_2+c_{12} \le 6$, we have $|S|=6$, $c_1=c_2=1$, $\la_1=\la_2=\la_3=2$ and $|G| \le 36$. By \eqref{eqn-la1la2} and \eqref{eqn-la3},
\begin{align*}
\la_1+\la_2=\frac{12}{|(H_1 \cap H_2) \sm \{1\}|}=4 &\Rightarrow |H_1 \cap H_2|=4, \\
\la_3=\frac{2}{|H_3 \sm (H_1 \cup H_2)|}=2 &\Rightarrow |H_3 \sm (H_1 \cup H_2)|=1.
\end{align*}
Since $c_1=c_2=1$, by \eqref{eqn-c1} and \eqref{eqn-c2},
\begin{equation}\label{eqn-inter}
\begin{aligned}
6 &\ge 4|(H_1 \cap H_3) \sm \{1\}|, \\
6 &\ge 4|(H_2 \cap H_3) \sm \{1\}|.
\end{aligned}
\end{equation}

If $H_1 \cap H_3 \ne \{1\}$ and $H_2 \cap H_3 \ne \{1\}$, by \eqref{eqn-inter}, we have $|H_1 \cap H_3|=|H_2 \cap H_3|=2$. Together with $|H_3 \sm (H_1 \cup H_2)|=1$, we have $|H_3|=4$. Since $H_1 \cap H_2 \cap H_3=\{1\}$, we have $|\Span\{H_1 \cap H_2, H_3\}|=16$. Thus, $|\Span\{H_1 \cap H_2, H_3\}| \big| |G|$ and $|S| \big| |G|$ implies $48 \big| |G|$. We derive a contradiction on the size of $G$.

If exactly one of $H_1 \cap H_3$ and $H_2 \cap H_3$ is $\{1\}$. We assume without loss of generality that $H_1 \cap H_3=\{1\}$ and $H_2 \cap H_3 \ne \{1\}$. By \eqref{eqn-inter}, we have $|H_2 \cap H_3|=2$. Together with $|H_3 \sm (H_1 \cup H_2)|=1$, we have $|H_3|=3$, which leads to a contradiction since $|H_2 \cap H_3| \nmid |H_3|$.

If $H_1 \cap H_3=H_2 \cap H_3=\{1\}$, by \eqref{eqn-la1}, \eqref{eqn-la2} and $|H_1 \cap H_2|=4$,
\begin{align*}
\la_1=\frac{8}{|H_1 \sm (H_2 \cup H_3)|}=2 &\Rightarrow |H_1 \sm (H_2 \cup H_3)|=4, |H_1|=8, \\
\la_2=\frac{8}{|H_2 \sm (H_1 \cup H_3)|}=2 &\Rightarrow |H_2 \sm (H_1 \cup H_3)|=4, |H_2|=8.
\end{align*}
Therefore, $|\Span\{H_1, H_2\}|=16$. Thus, $|\Span\{H_1, H_2\}| \big| |G|$ and $|S| \big| |G|$ implies $48 \big| |G|$. We derive a contradiction on the size of $G$.
\end{proof}

\section*{Acknowledgement}

Shuxing Li is supported by the Alexander von Humboldt Foundation. Robert Sch\"{u}ler is supported by DFG grant SCHU 1503/7. The authors wish to express their gratitude to the anonymous reviewers for their careful reading and very helpful suggestions.

\end{document}